\documentclass[12pt]{article}
\usepackage{fullpage,epsfig}
\usepackage{a4}
\usepackage[english]{babel}                 
\usepackage{graphicx,color}
\usepackage[colorinlistoftodos]{todonotes}

\usepackage[T1]{fontenc}
\usepackage{amsmath}                       
\usepackage{amssymb}                       
\usepackage{amsfonts}                      
\usepackage{amsthm}     
\usepackage{enumerate}                     

\newcommand{\RR}{\mathbb R}
\newcommand{\NN}{\mathbb N}

\newcommand{\cA}{\mathcal{A}}

\newcommand{\cL}{\mathcal{L}}

\newcommand{\cM}{\mathcal{M}}
\newcommand{\cH}{\mathcal{H}}

\newcommand{\cU}{\mathcal{U}}
\newcommand{\eps}{\varepsilon}

\newcommand{\supp}{\operatorname{supp}}

\newcommand{\mbf}[1]{\boldsymbol{#1}}
\newcommand{\To}{\longrightarrow}

\newcommand{\abs}[1]{\left\vert#1\right\vert}

\newcommand{\absn}[1]{\vert#1\vert}
\newcommand{\absb}[1]{\big\vert#1\big\vert}
\newcommand{\absB}[1]{\Big\vert#1\Big\vert}

\newcommand{\norm}[1]{\left\|#1\right\|}

\newcommand{\dist}[2] {\operatorname{dist}\left(#1;#2\right)}

\newcommand{\mysetr}[2] {\left\{#1\,\left|\,#2\right.\right\}}

\newcommand{\mysetb}[2] {\big\{#1\,\big|\,#2\big\}}

\newcommand{\be}{\begin{equation}}
\newcommand{\ee}{\end{equation}}
\newcommand{\bald}{\begin{aligned}}
\newcommand{\eald}{\end{aligned}}

\newtheorem{theorem}{Theorem}

\newtheorem{thm}[theorem]{Theorem}
\newtheorem{cor}[theorem]{Corollary}
\newtheorem{lem}[theorem]{Lemma}
\newtheorem{prop}[theorem]{Proposition}
\theoremstyle{definition}
\newtheorem{defn}[theorem]{Definition}
\newtheorem{ex}[theorem]{Example}
\theoremstyle{remark}
\newtheorem{rem}[theorem]{Remark}
\renewcommand{\proofname}{\bfseries{Proof}}
\renewenvironment{proof}[1][\proofname]{\par
  \normalfont
  \trivlist
  \item[\hskip\labelsep\itshape
    \bfseries{#1.}]\ignorespaces
}{%
  \qed\endtrivlist
}

\newlength{\alignskip}
\newcommand{\wstar}{{\stackrel{*}{\rightharpoonup}}\,}
\renewcommand{\d }{d}

\author{Barbora Bene\v{s}ov\'{a}\footnote{Department of Mathematics I, RWTH Aachen University, D-52056 Aachen, Germany} ,  Stefan Kr\"{o}mer\footnote{Math.~Inst., Universit\"{a}t zu K\"{o}ln, 50923 K\"{o}ln, Germany;  skroemer@math.uni-koeln.de} ,   Martin Kru\v{z}\'{\i}k\footnote{Institute of Information Theory and Automation of the ASCR, Pod vod\'{a}renskou
v\v{e}\v{z}\'{\i}~4, CZ-182~08~Praha~8, Czech Republic and
Faculty of Civil Engineering, Czech Technical
University, Th\'{a}kurova 7, CZ-166~ 29~Praha~6, Czech Republic}
}
\title{Boundary effects and weak$^*$ lower semicontinuity for signed integral functionals on $\mathrm{BV}$}%

\begin{document}
\date{}
\maketitle
\begin{abstract}
We characterize  lower semicontinuity of integral functionals with respect to weak$^*$ convergence in $\mathrm{BV}$, including  integrands whose negative part has  linear growth. In addition, we allow for  sequences without a fixed trace at the boundary. In this case,  both the integrand and the shape of the boundary play a key role.   This is made precise in our newly found condition -- quasi-sublinear growth from below at points of the boundary --  which   compensates for possible concentration effects generated by the sequence.     Our work extends  some recent results by J.~Kristensen and F.~Rindler (Arch.~Rat.~Mech.~Anal.~197 (2010), 539--598 and  Calc.~Var.~37 (2010), 29--62). 

\end{abstract}

\section{Introduction}\label{sec:intro}

 We consider an integral functional of the form
\be
	F(u):=\int_\Omega \d f(x,Du),~~u\in \mathrm{BV}(\Omega;\RR^M),
	\label{defF}
\ee
where $x$ is the variable of integration and
\[
	\d f(x,Du):=f(x,\nabla u(x)) \d x+f^\infty\Big(x,\frac{dDu}{d\abs{Du}}(x)\Big) \d \abs{D^su}(x).
\]
Here, we assume that the energy density $f:\overline{\Omega}\times \RR^{M\times N}\to \RR$ has a recession function which is denoted by $f^\infty$ (see \eqref{f2} below); $Du$ is the distributional derivative of $u:\Omega\to \RR^M$, a finite matrix-valued Radon measure on $\Omega\subset\RR^N$, and
$$
	Du=(\nabla u) \cL^N+D^su
$$
is its Radon-Nikod\'{y}m decomposition with respect to the Lebesgue measure $\cL^N$, with $D^su$ and $\nabla u$ denoting the singular part and the density of the absolutely continuous part, respectively. Moreover, we used that $\frac{\d Du}{\d \abs{Du}}(x)=\frac{\d D^su}{\d \abs{D^su}}(x)$ for $\abs{D^s u}$-a.e.~$x$; here $\abs{\mu}$ denotes the total variation of the measure $\mu$ and 
$\mu=\frac{\d\mu}{\d\abs{\mu}}\abs{\mu}$ is its the polar decomposition, i.e.,
$\frac{\d\mu}{\d\abs{\mu}}$ is the density of $\mu$ with respect to $\abs{\mu}$ (Besicovitch derivative) which satisfies $\absb{\frac{\d\mu}{\d\abs{\mu}}(x)}=1$ for $\abs{\mu}$-a.e.~$x$. 

The aim of this paper is to give a precise characterization of weak$^*$ lower semicontinuity of \eqref{defF} in $\mathrm{BV}$ without prescribing fixed Dirichlet boundary data or assuming that $f$ is bounded from below. For $f$  bounded from below, weak$^*$ lower semicontinuity and relaxation of $F$ were examined  in \cite{FoMue92a,FoMue93a}, even including explicit dependence on $u$ (not just $Du$). Nevertheless, in order to characterize the appropriate generalizations of gradient Young measures in $\mathrm{BV}$ \cite{KriRi10b} (gradient DiPerna-Majda measures) it is necessary to allow also for a linear growth of the negative part of the integrand. In this case, weak$^*$ lower semicontinuity and relaxation is treated in \cite{KriRi10a,KriRi10b} but the results are valid only for sequences with fixed trace on the boundary of $\Omega$, or if a term penalizing jumps at the boundary is added to the functional (for related results in $W^{1,p}$ with $p>1$ see \cite{FoMuePe98a,KaKru08a}). In particular, it is not possible to use them to show attainment of minimizers of $F$ without prescribed Dirichlet data, and the characterization of the generalized Young measures generated by gradients (distributional derivatives of functions in $\mathrm{BV}$) with free function values on the boundary.
Such characterizations are known in $W^{1,p}$ for $p>1$ \cite{KroeKru13a,Kru10a}. 

Clearly, if no Dirichlet boundary condition is prescribed, boundary effects play a role in characterizing the weak$^*$ lower semicontinuity. To see this, we consider the following simple example.
\begin{ex}
 Take $\varphi\in W^{1,1}_0(B_1(0);\RR^M)$, where $B_r(0)$ is a  ball in $\RR^N$  with the radius $r>0$   centered at $0$, and extend it by zero to the whole $\RR^N$. 
Define for $x\in\RR^N$ and $k\in\NN$  $\varphi_k(x)=k^{n-1}\varphi(kx)$, i.e., $\varphi_k \wstar 0$ in $\mathrm{BV}(B_1(0);\RR^M)$ and  consider a smooth  domain
 $\Omega\in\RR^N$ such that $0\in\partial\Omega$, $\nu_0$ is the outer unit normal to $\partial\Omega$ at $0$. Moreover, take a  function  $f:\RR^{M\times N}\to\RR$ to be  positively $1$-homogeneous, i.e., $f(\alpha \xi)=\alpha f(\xi) $ for all $\alpha\ge 0$. If $F$ from \eqref{defF}    is weakly* lower semicontinuous then 
\begin{eqnarray}\label{eq:intro}
0&=&F(0)\le \liminf_{k\to\infty} \int_\Omega f(\nabla\varphi_k(x))\,\d x= \liminf_{k\to\infty} \int_{B_{1/k}(0)\cap\Omega}f(\nabla\varphi_k(x))\,\d x\\ &=& \liminf_{k\to\infty}\int_{B_{1/k}(0)\cap\Omega}k^n f(\nabla \varphi(kx))\,\d x= \int_{B_1(0)\cap \{y\in\RR^N;\ \nu_0\cdot y<0\}} f(\nabla\varphi(y))\,\d y\ .\nonumber
\end{eqnarray}
Thus, we see that 
\begin{equation}
0\le \int_{B_1(0)\cap \{y\in\RR^N;\ \nu_0\cdot y<0\}} f(\nabla\varphi(y))\,\d y
\label{example}
\end{equation}
for all $\varphi\in W^{1,1}_0(B_1(0);\RR^M)$  forms a necessary condition for weak$^*$ lower semicontinuity of $F$ whenever $f$ is positively $1$-homogeneous.  
\end{ex}

Within this paper, we prove that, in case of a smooth boundary, condition \eqref{example} together with quasiconvexity of $f$ is indeed also \emph{sufficient} for weak* lower semi-continuity, however with $f$ replaced by its recession function if the former is not homogeneous of degree one; cf. Theorem \ref{thm:BVwlsc} below. In case $\Omega$ is not smooth enough, we suitably generalize condition \eqref{example} and introduce the notion of {\it quasi-sublinear growth from below} (cf. Definition \ref{def:qslb}), which is central in our work. As for the sufficency of quasi-convexity and quasicovenxity of quasi-sublinear growth from below for weak* lower semi-continuity, our results can then cope with domains with a rather irregular boundary, for which a jump term integrated over the boundary would not be well defined. For the necessity part, we need to work on an extension domain so that we can rely on compact embedding results.

Nevertheless, it should be stressed that even if $\Omega$ is of such smoothness that we may extend the function $u$ entering \eqref{defF} by zero to $\mathrm{BV}(\Omega'; \RR^M)$ to some regular domain $\Omega' \supset \Omega$, studying weak* lower semicontinuity of \eqref{defF} is \emph{not} equivalent to studying the extended problem due to the additional contribution from the jump term over the boundary. To illustrate this, consider the following example:

\begin{ex}[following \cite{KriRi10a}, \cite{BaCheMaSa13a}]
\label{rem-qcANDqslb}
Choose $\Omega=(0,1)$ and define  $u_n:=\chi_{(0,\frac{1}{n})}$ so that $Du_n=-\delta_{\frac{1}{n}}$. Further let us choose and $f(x,\xi):=\xi$; 
then $F$ is a linear functional and $F(u_n)=-1$ for all $n$, but $u_n \wstar 0$ to $0$ in $\mathrm{BV}((0,1))$ and $F(0)=0>-1$.

Nevertheless, if we enlarged the domain to, say, $(-1,1)$ and extended $u_n$ by zero, then we would get $F(u_n)=0$, giving weak* lower semicontinuity along this sequence.

In the context of our characterization, we notice that while $f$ is linear and thus quasiconvex in its second variable, it does not satisfy our new condition of quasi-sublinear growth from below at $x_0=0$.

\end{ex}

We also mention that results concerning regularity and uniqueness\footnote{up to additive constants; in general, even with a jump term on the boundary penalizing the distance to prescribed Dirichlet data, these cannot be avoided because the recession function is never strictly convex} of minimizers are available, although 
only for convex and coercive integrands assuming a form of very strong ellipticity \cite{BeSc13a}.

The plan of the paper is as follows.  Our main result, Theorem~\ref{thm:BVwlsc} is stated in  Section~\ref{sec:results} preceded with necessary definitions and notation. In particular, Definition~\ref{def:qslb} describes quasi-sublinear growth from below. Various useful variants of this condition are discussed in Section~\ref{sec:variants} and the Appendix.
As a preparatory part for the proof of Theorem~\ref{thm:BVwlsc},  Section~\ref{sec:dec} deals with a suitable decomposition of sequences in $\mathrm{BV}$. Finally, a proof of the main result is given in Section~\ref{sec:necsuff}. 

\section{Main results}\label{sec:results}

In the bulk of this article, the domain $\Omega\subset\RR^N$ can be any open and bounded set.   If we additionally need  $\Omega$ to allow for a compact embedding of $\mathrm{BV}(\Omega;\RR^M)$ into $L^1(\Omega;\RR^M)$ we write $\Omega\in E(\RR^N)$. An extension domain (or even a Lipschitz domain) serves as an example of a set belonging to $E(\RR^N)$.
If even more  regularity of the  boundary is required, this is stated on the spot; in that case, we usually 
need a boundary of class $C^1$.  Here and in the sequel, $\cM(\Omega;\RR^{M\times N})$ means the set of $\RR^{M\times N}$-valued Radon measures on $\Omega$ and   $\mathrm{BV}(\Omega;\RR^M)$ denotes the standard space of maps $\Omega\to\RR^M$ which have bounded variation; cf.~\cite{AmFuPa00B} for details.  As the  weak$^*$ convergence in $\mathrm{BV}(\Omega;\RR^M)$ is a central notion of our analysis we recall that $(u_n)_{n\in\NN}\subset \mathrm{BV}(\Omega;\RR^M)$ converges weakly$^*$ to $u\in\mathrm{BV}(\Omega;\RR^M)$ if 
$u_n\to u$ strongly in $L^1(\Omega;\RR^M)$ and $Du_n\stackrel{*}{\rightharpoonup} Du$ in $\cM(\Omega;\RR^{M\times N})$; see \cite[Def.~3.11]{AmFuPa00B}. If $\mu\in\cM(\Omega)$ then $|\mu|$ denotes its total variation (norm), i.e., $|\mu|:=\sup|\int_\Omega f\,d\mu|$ where the supremum is taken over all $f\in C_0(\Omega)$ such that $\|f\|_{C_0(\Omega)}=1$. Here, $C_0(\Omega)$ is the space of continuous functions vanishing at the boundary of $\Omega$.
 Moreover, $W^{1,p}(\Omega;\RR^M)$ and $L^p(\Omega;\RR^M)$ stand for classical Sobolev and Lebesgue spaces, respectively. 
Throughout the paper, we assume that
\be
	f:\overline{\Omega}\times \RR^{M\times N}~~\text{is continuous}\tag{f:0}\label{f0}
\ee
with at most linear growth, i.e., there is a constant $C\geq 0$ such that
\be
	\abs{f(x,\xi)}\leq C (\abs{\xi}+1)~~\text{for every $(x,\xi)\in \overline{\Omega}\times \RR^{M\times N}$,}\tag{f:1}\label{f1}
\ee
and $f$ admits a recession function in the following sense:
\be
\begin{aligned}
	&f^\infty(x,\xi):=\lim_{\underset{\underset{y\to x}{t\to+\infty}}{\eta\to\xi}} \frac{f(y,t\eta)}{t}~~\text{exists for every $(x,\xi)\in \overline{\Omega}\times (\RR^{M\times N}\setminus \{0\}$),}
\end{aligned}	\tag{f:2}\label{f2}
\ee
Note that by definition, $f^\infty$ is continuous on $\overline{\Omega}\times (\RR^{M\times N}\setminus \{0\})$ and
positively $1$-homogeneous in $\xi$. 

\begin{rem} The restriction of $F$ to $W^{1,1}(\Omega;\RR^M)$ reads
$$
	\tilde{F}(u):=\int_\Omega f(x,\nabla u)\,\d x,~~u\in W^{1,1}(\Omega;\RR^M),
$$
and $F$ is the natural extension of $\tilde{F}$ to $\mathrm{BV}$. 
\end{rem}

Our main result provides a characterization of sequential lower semicontinuity of $F$ with respect to weak$^*$-convergence in $\mathrm{BV}$, in the usual sense recalled below. It is natural to expect that this characterization will be linked to the well-known quasiconvexity condition in the sense of Morrey \cite{Mo52a} as given in Definition \ref{def:qc}. In addition to that, we need an additional property of $f$ to prevent negative contributions of sequences concentrating at the boundary of the domain; cf. Definition \ref{def:qslb}.

\begin{defn}[w$^*$lsc]\label{def:wlsc}
We say that the functional $F:\mathrm{BV}(\Omega;\RR^M) \to \RR$ is sequentially \emph{weakly$^*$ lower semicontinuous (w$^*$lsc)} in $\mathrm{BV}(\Omega;\RR^M)$ if 
$$
\liminf_{n \to \infty} F(u_n)\geq F(u)
$$
for every sequence $(u_n)\subset \mathrm{BV}(\Omega;\RR^M)$ and such that $u_n\wstar u$ in $\mathrm{BV}$.
\end{defn}
\begin{defn}[Quasiconvexity]\label{def:qc}
A function $g:\RR^{M\times N}\to \RR$ is called \emph{quasiconvex at $\xi\in \RR^{M\times N}$} if
\begin{align*}
	\int_{B} \big(g(\xi+\nabla \varphi(y))-g(\xi)\big) \, \d y\geq 0
	~~\text{for every $\varphi\in W_0^{1,\infty}(B;\RR^M)$},
\end{align*}
where $B$ denotes the unit ball in $\RR^N$ (or, equivalently, any arbitrary fixed bounded Lipschitz domain). We say that $g$ is \emph{quasiconvex} (qc) if it is quasiconvex at every $\xi\in \RR^{M\times N}$.
\end{defn}
\begin{defn}[Quasi-sublinear growth from below]\label{def:qslb}
We say that a function $g:\RR^{M\times N}\to \RR$ is \emph{quasi-sublinear from below} (\emph{qslb}) at a point
$x_0\in \overline{\Omega}$ if
\[
\bald
	&\bald[t]
	&\text{for every $\eps>0$, there exist $\delta=\delta(\eps)>0$, $C=C(\eps)\in \RR$ s.t.}\\
	&\int_{\Omega\cap B_\delta(x_0)} g(\nabla v(x))\, \d x
	~\geq~
	-\eps \int_{\Omega\cap B_\delta(x_0)} \abs{\nabla v(x)}\, \d x-C 
	\eald  \\
	&\text{for every $v\in W^{1,1}(B_\delta(x_0)\cap\Omega;\RR^M)$ with $v=0$ near $\partial B_\delta(x_0)$.}
\eald
\]
\end{defn}
\begin{rem}
Definition~\ref{def:qslb} is a straightforward extension to the case $p=1$ of the corresponding condition for $p>1$, (3.2) in \cite{Kroe10b}.  If $\Omega$ is an extension domain, the class of test functions can be replaced by
$v\in W_0^{1,1}(B_\delta(x_0);\RR^M)$; essentially, we want $v$ to vanish in a neighborhood of
(or have zero trace on) the ``interior'' boundary $\partial B_\delta(x_0)\cap \Omega$, while being free on $\partial\Omega$. Whenever necessary, $v$ is understood to be extended 
by zero to $\Omega\setminus B_\delta(x_0)$. 
\end{rem}
\begin{rem}
Quasi-sublinear growth from below is a local condition at the point $x_0$ in the sense that if it holds for some $\delta$, then also for all $\delta'\leq \delta$, because the class of test functions becomes smaller, and the difference in the integral on the left hand side is given by $\int_{\Omega\cap (B_{\delta}(x_0)\setminus B_{\delta'}(x_0))} g(0)\,dx$, a constant that can be absorbed by $C$.
If $\partial\Omega$ is of class $C^1$ near $x_0$,
and $g$ is regular enough, it is possible to rescale the domain of integration to unit size and pass to the limit as $\delta\to 0$. Doing so reduces the quasi-sublinear growth from below of $g$ to the following, equivalent condition (details are given in Proposition~\ref{prop:qslb}):
\be\label{qslb-limit1}
\bald
	&\bald[t]
	&\text{For every $\eps>0$, there exists $C=C(\eps)\geq 0$ such that}\\
	&\int_{D_{x_0}} g(\nabla \varphi(y))\, \d y
	~\geq~
	-\eps \int_{D_{x_0}} \abs{\nabla \varphi(y)}\, \d y-C 
	\eald & \\
	&\text{for all $\varphi\in W_0^{1,1}(B_1(0);\RR^M)$, where $D_{x_0}:=\{y\in B_1(0)\mid y\cdot \nu_{x_0}<0\}$,}
\eald
\ee
the half-ball opposite of the outer normal $\nu_{x_0}$ to $\partial\Omega$ at $x_0$. 
This corresponds to the case $p=1$ in Theorem 1.6 (ii) in \cite{Kroe10b} 
and $1$-quasisubcritical growth from below at a boundary point as defined in \cite{KroeKru13a}.
\end{rem}
\begin{rem}
If one replaces the gradients $\nabla v$ by arbitrary integrable matrix fields in $L^1$ in one of the versions of quasi-sublinear growth from below (which makes it more restrictive due to the then larger class of test functions), it turns into ``standard'' sublinear growth from below, in the sense that for each $\eps>0$, there is $C_\eps$ such that $g(\xi)\geq -\eps\abs{\xi}-C_\eps$ for all $\xi\in \RR^{M\times N}$. The latter is equivalent to $\liminf_{\abs{\xi}\to\infty} \frac{g(\xi)}{\abs{\xi}} \geq 0$.
\end{rem}
With these definitions at hand, we formulate our main result: 
\begin{thm}\label{thm:BVwlsc}
Let $\Omega\in E(\RR^N)$  and assume that \eqref{f0}--\eqref{f2} hold. Then the functional $F$ 
defined in
\eqref{defF} is w$^*$lsc if and only if its integrand $f$ simultaneously satisfies the following:
\begin{enumerate}
\item[(i)] $f(x_0,\cdot)$ is quasiconvex for a.e.~$x_0\in \Omega$;
\item[(ii)] $f(x_0,\cdot)$ is of quasi sub-linear grwoth from below at $x_0$ for every $x_0\in \partial \Omega$.
\end{enumerate}
Moreover, $f$ can be replaced with $f^\infty$ in (ii), and if $\partial\Omega$ is of class $C^1$, then (ii) holds if and only if for every $x_0\in\partial\Omega$,
\be
\label{qslb-limit2}
\int_{D_{x_0}}
f^\infty(x_0,\nabla \varphi(y))\,dy \geq 0\quad\text{for all $\varphi\in W_0^{1,1}(B_1(0);\RR^M)$,}
\ee
where $D_{x_0}:=\{y\in B_1(0)\mid y\cdot \nu_{x_0}<0\} $, with the outer normal $\nu_{x_0}$ to $\partial\Omega$ at $x_0$.
\end{thm}
A detailed proof of the if and only if characterization of weak* lower semicontinuity is the content of Section~\ref{sec:necsuff} while, the second part of the theorem is content of Proposition \ref{prop:qslb}. For convenience, we sketch the main idea of the proof of the charachterization of weak* lower semi-continuity here. 

\begin{proof}[Idea of the proof]
The necessity of the quasiconvexity is standard while the necessity of the quasi-sublinear growth from below follows by a contradiction argument.

As for the sufficiency, we assume, for simplicity, in this sketch that the sequence $(u_n)$ from Definition \ref{def:wlsc} satisfies $u_n \wstar 0$ in $\mathrm{BV}(\Omega; \mathbb{R}^M)$. Then proof then relies on the observation that we can ``separate'' the behavior of $(u_n)$ from in the interior of the domain $\Omega$ and on its boundary. Indeed, by the local decomposition lemma (shown in Section \ref{sec:dec}) we may write $(u_n)$ as a sum of a sequence that is supported inside $\Omega$ and a sequence that is purely concentrating at the boundary (i.e.~it is supported in a vanishingly small neighbourhood of $\partial \Omega$. Moreover, the decomposition is such that we can treat this two sequences as essentially independent; i.e. it suffices to show weak* lower semicontinuity along each of the sequences separately (cf.~also Proposition \ref{prop:fdecloc}).

Now, for the sequence supported inside the domain, we are in the situation of 
\cite{KriRi10b} (since all the members of the sequence have fixed--zero--Dirichlet boundary data). For the sequence that is purely concentrating on the boundary, we realize that along such sequences functionals bounded from below are weakly lower 
semicontinuous; cf. Proposition \ref{prop:lscbndconc}. Note that we need only the functional to have a lower bound, the density function may be unbounded. Finally, this lower bound is, roughly, provided by our quasi-sublinear growth from below condition from Definition \ref{def:qslb}. 
\end{proof}

\begin{rem}
Condition \eqref{qslb-limit2} is closely related but not equivalent to boundary quasiconvexity of $f^\infty(x_0,\cdot)$ at the zero matrix in direction $\nu_{x_0}$. For comparison: Boundary quasiconvexity at a matrix $\xi\in \RR^{M\times N}$ as defined in \cite{Spre96B,MieSpre98a}\footnote{Originally, boundary quasiconvexity at critical points was introduced by  Ball \& Marsden in \cite{BaMa84a} as a necessary condition for strong local minima}, requires that there exists a matrix%
\footnote{Actually, the original definition of boundary quasiconvexity in \cite{Spre96B,MieSpre98a} calls for the existence of a vector $q\in \RR^M$ playing the role of $A \nu_{x_0}$ in the boundary integral 
$\int_{\Gamma_{x_0}} \varphi(y) \cdot A \nu_{x_0} d\cH^{N-1}(y)$, where $\Gamma_{x_0}=\{y\in \bar{B}_1\mid y\cdot \nu_{x_0}=0 \}$ is the part of the boundary of $D_{x_0}$ where the trace of $v$ is free, which is equal to the  right hand side in \eqref{qcb} by integration by parts. Yet, expressing the right hand side in terms of the matrix $A$ and a volume integral makes it obvious that \eqref{qcb} is in fact a generalized notion of convexity.}
$A\in \RR^{M\times N}$ (only depending on $\xi$ and $x_0$) such that
\be
\label{qcb}
\bald
	&\int_{D_{x_0}}
	f^\infty(x_0,\nabla \varphi(y)+\xi)\, \d y - \int_{D_{x_0}}
	f^\infty(x_0,\xi)\,dy\\
	&\qquad \geq \int_{D_{x_0}} A:\nabla \varphi(y)\, \d y 
\quad\text{for all $\varphi\in W_0^{1,1}(B_1;\RR^M)$.}
\eald
\ee
Since $f^\infty(x_0,0)=0$, \eqref{qslb-limit2} coincides with
with \eqref{qcb} at $\xi=0$ \emph{if} $A=0$ is admissible, but in general, \eqref{qslb-limit2} is more restrictive.
\end{rem}

\begin{rem}
A further intuition on the quasi-sublinear growth from below condition can be gained from Proposition \ref{prop:qslb} (below) where we prove that for \emph{any interior point $x$ in $\Omega$} this condition is equivalent to the quasiconvexity of the recession function in $0$; in general this is a weaker condition that standard quasiconvexity. However, along a purely concentrating sequence in $x$ one may, roughly, ``replace'' $f$ by its recession function and the weak* limit of such a sequence is necessarily 0. Nevertheless, one should be very careful with such an intuition at the boundary; cf. Exammple \ref{rem-qcANDqslb} where it is shown that at boundary points quasi-sublinear growth from below is \emph{not} implied by quasiconvexity in general.
\end{rem}

\begin{rem}
In Example \ref{rem-qcANDqslb} we saw a linear functional that does not satisfy the quasi-sublinearity from below conditions. However, functions that satisfy \eqref{qslb-limit2} even with equality can be constructed: Let $t\in\RR^N$ be a vector perpendicular to $\nu_{x_0}$ and  $f:\RR^{M\times N}\to\RR$, $f(\xi)=f^\infty(\xi):=\xi{:}a\otimes t$ with $a\in\RR^M$. Thus, $f$ defines a so-called linear Null Lagrangian at the boundary. We refer to 
\cite{KaKroeKru12a} for Null Lagrangians at the boundary of higher order. 
\end{rem}

\section{Variants of quasi-sublinear growth from below}\label{sec:variants}

In this section, we collect several conditions equivalent to quasi-sublinear growth from below, useful either for technical purposes (particularly \eqref{qslb-unfrozen} and \eqref{qslb-unfrozenBV}) or (somewhat) easier to check for a given integrand.
Moreover, we prove in Proposition \ref{prop:qslb} the second part of Theorem \ref{thm:BVwlsc}; i.e. that $f(x_0,\cdot)$ turns out to be of quasi sub-linear growth from below if and only if its recession function is,
and the latter is equivalent to \eqref{qslb-limit2} in case of $C^1$-boundary. 

First, we realize that for quasi sub-linear growth from below of the recession function the constant $C$ can be dropped (due to $1$-homogeneity):
\be \label{qslb-finfty}
\bald
	&\bald[t]
	&\text{for every $\eps>0$, there exists $\delta=\delta(\eps)>0$ such that}\\
	&\qquad\int_{\Omega\cap B_\delta(x_0)} f^\infty(x_0,\nabla v(x))\,dx
	~\geq~
	-\eps \int_{\Omega\cap B_\delta(x_0)} \abs{\nabla v(x)}\,dx 
	\eald &\\
	&\text{for every $v\in W^{1,1}(B_\delta(x_0)\cap\Omega;\RR^M)$ with $v=0$ near $\partial B_\delta(x_0)$}.
\eald
\ee
On the other hand, it is also possible to replace the ``frozen'' $x_0$ in the first argument of $f$ by the variable of integration, which yields a condition more convenient for proving weak$^*$ lower semicontinuity:
\be \label{qslb-unfrozen}
\bald
	&\bald[t]
	&\text{For every $\eps>0$, there exists $\delta>0$ and $C\geq 0$ such that}\\
	&\qquad \int_{\Omega\cap B_\delta(x_0)} f(x,\nabla v(x))\,dx
	~\geq~
	-\eps \int_{\Omega\cap B_\delta(x_0)} \abs{\nabla v(x)}\,dx - C
	\eald \\
	&\text{for every $v\in W^{1,1}(\Omega\cap B_\delta(x_0);\RR^M)$ with $v=0$ near $\partial B_\delta(x_0)$}.
\eald
\ee
This variant is more natural if $f^\infty$ is not continuous in its first variable (and in that case, it is no longer equivalent to qslb of $f^\infty(x_0,\cdot)$ in the sense of Definition~\ref{def:qslb}).

\begin{rem}\label{rem:qslbBV}
Using the density of $W^{1,1}$ in $\mathrm{BV}$ with respect to area-strict convergence ($\langle\cdot\rangle$-strict convergence)\footnote{see Section 2.2 in \cite{KriRi10a}; by definition, a sequence $(u_n)\subset \mathrm{BV}$ converges $\langle\cdot\rangle$-strictly
to $u\in \mathrm{BV}$ if $u_n\wstar u$ in $\mathrm{BV}$ and
$
	\int_\Omega \sqrt{1+\abs{\nabla  u_n}^2}\,dx+\abs{D^s u_n}(\Omega)
	\to 
	\int_\Omega \sqrt{1+\abs{\nabla  u}^2}\,dx+\abs{D^s u}(\Omega)
$
}, together 
with the associated variant of Reshetnyak's continuity theorem (\cite[Theorem 5]{KriRi10a} and, more general, \cite[Theorem 1]{RiSha13a}),
all our variants of quasi-sublinear growth from below have equivalent versions extending the class of test functions from $W^{1,1}$ to $\mathrm{BV}$. Most importantly for our purposes, \eqref{qslb-unfrozen} is equivalent to the following:
\be \label{qslb-unfrozenBV}
\bald
	&\bald[t]
	&\text{For every $\eps>0$, there exist $\delta=\delta(\eps)>0$, $C=C(\eps)\geq 0$ such that}\\
	&\qquad\int_{\Omega\cap B_\delta(x_0)} 
	df(x,Dv)
	~\geq~
	-\eps \abs{Dv} (\Omega\cap B_\delta(x_0)) - C
	\eald \\
	&\text{for every $v\in \mathrm{BV}(\Omega\cap B_\delta(x_0);\RR^M)$ with $v=0$ near $\partial B_\delta(x_0)$}.
\eald
\ee
\end{rem}
The relationship between the variants of quasi-sublinear growth from below, and their link to quasiconvexity for interior points $x_0$, can be summarized as follows:  
\begin{prop}\label{prop:qslb}
Suppose that $f$ satisfies \eqref{f0}--\eqref{f2}.
Then the following holds:
\begin{enumerate}
\item[(a)] For every $x_0\in\overline{\Omega}$,
$$
	\text{$f(x_0,\cdot)$ is qslb at $x_0$}
	~~\Longleftrightarrow~~
	\text{$f^\infty(x_0,\cdot)$ is qslb at $x_0$}
	~~\Longleftrightarrow~~
	\eqref{qslb-finfty}
	~~\Longleftrightarrow~~
	\eqref{qslb-unfrozen}.
$$
\item[(b)]
For every $x_0 \in \Omega$,
$$	
	\text{$f(x_0,\cdot)$ qc at $0$}
	~~\Longrightarrow~~
	\text{$f^\infty(x_0,\cdot)$ qc at $0$}
  ~~\Longleftrightarrow~~ 
  \text{$f^\infty(x_0,\cdot)$ qslb at $0$}.
$$
\item[(c)]
For every $x_0 \in\partial\Omega$ such that $\partial \Omega$ is of class $C^1$ near $x_0$,
$$	
	\text{$f^\infty(x_0,\cdot)$ qslb at $x_0$}
	~~\Longleftrightarrow~~
	\text{\eqref{qslb-limit1} for $g:=f^\infty(x_0,\cdot)$}
	~~\Longleftrightarrow~~
	\eqref{qslb-limit2}.
$$
\end{enumerate}
\end{prop}
The detailed proof is given in the appendix but we provide a short idea of the proof here.

\begin{proof}[Idea of the proof]
The proof of the first equivalence in (a) is based on Proposition \ref{prop:finfty} (given below) which assures that for large values of the second variable $f(x_0, \cdot)$ and $f^\infty(x_0, \cdot)$ can be interchanged with only a small error as well as on realizing that, essentially, only these large values of the second variable play a role in the defintion of quasi-sublinear growth from below. As for the second two equivalences in (a), i.e. the transition to the ``unfrozen'' variants of quasi-sublinear growth from below, we rely in uniform continuity of $f$ as well as $f^\infty$. 

The first implication in (b) is well known (cf. \cite{FoMue93a}) while the equivalence is based on a change of variables argument.

The proof of the equivalences in (c) is based on a changes of variables argument and (locally) flattening the boundary.
\end{proof}

\begin{prop}\label{prop:finfty}
Assume that \eqref{f0}--\eqref{f2} hold. Then there exists a bounded, non-increasing function 
$\mu:[0,\infty)\to [0,\infty)$ with $\lim_{t\to+\infty} \mu(t)=0$ such that
\be\label{finfty-1}
	\abs{f(x,\xi)-f^\infty(x,\xi)}\leq \mu(\abs{\xi})(1+\abs{\xi})\quad\text{for every $(x,\xi)\in \overline{\Omega}\times \RR^{M\times N}$}.
\ee
\end{prop}
\begin{proof}
It suffices to show that
$$
	\mu(t):=\sup_{x\in \overline{\Omega},~\abs{\xi}\geq t} 
	\frac{\absn{f(x,\xi)-f^\infty(x,\xi)}}{1+\abs{\xi}}
	\underset{t\to+\infty}{\To}0.
$$	
Suppose by contradiction that there exists an $\eps>0$ and a sequence $(x_n,\xi_n)\in \overline{\Omega}\times \RR^{M\times N}$ with
$\abs{\xi}_n\to \infty$ such that
\be\label{pfi-2}
	\frac{\absn{f(x_n,\xi_n)-f^\infty(x_n,\xi_n)}}{1+\abs{\xi_n}}\geq \eps.
\ee
Passing to a subsequence (not relabeled), we may assume that
$$
	x_n\to x~~\text{and}~~\zeta_n:=\frac{\xi_n}{\abs{\xi_n}}\to \zeta
$$
for some $x\in \overline{\Omega}$ and $\zeta\in \RR^{M\times N}$ with $\abs{\zeta}=1$. Since $f^\infty$ is positively $1$-homogeneous in its second variable,
we see that
$$
\begin{aligned}
	\frac{\absn{f(x_n,\xi_n)-f^\infty(x_n,\xi_n)}}{1+\abs{\xi_n}}
	&=
	\absB{\frac{1}{\abs{\xi_n}}f(x_n,\zeta_n\abs{\xi_n})-f^\infty(x_n,\zeta_n)}
	\frac{\abs{\xi_n}}{1+\abs{\xi_n}}\\
	&\leq \absB{\frac{1}{\abs{\xi_n}}f(x_n,\zeta_n\abs{\xi_n})-f^\infty(x,\zeta)}+
	\absn{f^\infty(x,\zeta)-f^\infty(x_n,\zeta_n)}
\end{aligned}
$$
By \eqref{f2} and the uniform continuity of $f^\infty$ on the compact set $\overline{\Omega}\times S^{MN-1}$ (the latter denoting the unit sphere in $\RR^{M\times N}$), this converges to zero as $n\to\infty$, contradicting \eqref{pfi-2}.
\end{proof}
\section{Local decomposition results}\label{sec:dec}

Our proof of Theorem \ref{thm:BVwlsc} heavily realies on the ``local decomposition'' Lemma \ref{lem:decloc} given in this section. This lemma is, in a way, related to the well-known decomposition lemma  in $W^{1,p}$ ($p > 1$), that separates oscillations from concentrations (see \cite{AcFu84a,Kri94a,FoMuePe98a}), because it decomposes a weakly* converging sequence into a sum of sequences with localized support. The local decompositions lemma given here is an adaptation of Lemma 2.6 in \cite{Kroe10b} for $p=1$.

The following notion turns out to be useful.
\begin{defn}\label{def:charge}
Given a sequence $(u_n)\subset \mathrm{BV}(\Omega;\RR^M)$ and a closed set $K\subset \overline{\Omega}$, we say that \emph{$(D u_n)$ does not charge $K$}, if $\abs{Du_n}$ is tight in $\bar{\Omega}\setminus K$, i.e.,
$$
	\sup_{n\in\NN} \abs{D u_n}\big((K)_\delta\cap \Omega \big)\underset{\delta\to 0^+}{\To} 0.
$$
Here, $(K)_{\delta}:=\bigcup_{x\in A}B_\delta(x)$ denotes the open $\delta$-neighborhood of $K$ in $\RR^N$.
\end{defn}

\begin{lem}[local decomposition in $\mathrm{BV}$]
\label{lem:decloc}
Let $\Omega\subset \RR^N$ be open and bounded
and let $K_j\subset \overline{\Omega}$, $j=1,\ldots,J$, be a finite family of compact sets such that $\overline{\Omega}\subset \bigcup_j K_j$.
Then for every bounded sequence $(u_n)\subset \mathrm{BV}(\Omega;\RR^M)$ 
with
$u_n\to 0$ in $L^1(\Omega;\RR^M)$,
there exists a subsequence $(u_n)$ (not re-labelled) which can be decomposed as
$$
	u_{n}=u_{1,n}+\ldots+u_{J,n},
$$
where for each $j\in\{1,\ldots, J\}$, $(u_{j,n})_n$ is a bounded sequence in $\mathrm{BV}(\Omega;\RR^M)$ converging to zero in $L^1$ such that the following two conditions hold for every $j$:
\begin{align*}
\begin{alignedat}[c]{2}
	&\text{(i)}~~&&
	\{u_{j,n}\neq 0\}\subset \{u_n\neq 0\},~\overline{\{u_{j,n}\neq 0\}}
	\subset (K_j)_{\frac{1}{n}}
	\setminus \textstyle{\bigcup_{i<j}}(K_i)_{\frac{1}{2n}},\\
	&&& \text{$\abs{D u_{j,n}}\leq \abs{D u_n}+\tfrac{1}{n}\mathcal{L}^N$ as measures;}
	\\[\alignskip]
	&\text{(ii)}~~ &&\text{$(Du_{j,n})$ does not charge 
	$\textstyle{\bigcup_{i<j}}K_i$.}
\end{alignedat}
\end{align*}
Moreover, if $\partial \Omega$ is Lipschitz and each $u_n$ has vanishing trace on $\partial \Omega$, this is inherited by  $u_{j,n}$. Above, $(K_j)_{\frac{1}{n}}$ denotes the open $\frac{1}{n}$-neighborhood of $K_j$ in $\RR^N$ as before.
\end{lem}%
Apart from replacing $W^{1,1}$ with $\mathrm{BV}$ and some straightforward changes in notation for the distributional derivatives, the proof of Lemma 2.6 in \cite{Kroe10b} can be closely followed. The details are given below for convenience of the reader.
\begin{proof}[Proof of Lemma~\ref{lem:decloc}]

It suffices to discuss the case $J=2$ since the general case follows by iterating the argument. 
For every $n\in \NN$ choose a function $\varphi_n\in
C_c^\infty((K_1)_{\frac{1}{n}};[0,1])$ such that
$\varphi_n=1$ on $(K_1)_{\frac{1}{2n}}$, and define
$$
	v_{k,n}:=\varphi_n u_k,~~\text{whence}~
	D v_{n}=\varphi_n D u_k+u_k\otimes \nabla \varphi_n \cL^N.
$$
Note that since $u_k\to 0$ in $L^1$, 
$$
	u_k\otimes \nabla \varphi_n\underset{k\to\infty}{\To}0
	~~\text{in $L^1(\Omega;\RR^{M\times N})$ for every fixed $n$.}
$$
So, choose a subsequence $k(n)$ such that $\int_\Omega \absb{u_{k(n)}\otimes \nabla \varphi_{n}}\,dx
	\leq \frac{1}{n}$; in order to simplify the notation, we do not use the relabelling of the subsequence in the following; i.e. we assume that
\begin{align}\label{ldecloc-1a}
	\int_\Omega \absb{u_{n}\otimes \nabla \varphi_{n}}\,dx
	\leq \frac{1}{n}~~\text{for every $n \in \mathbb{N}$}
\end{align}
and set $v_n=v_{k(n),n}$.

In addition, inductively for $m\in\NN$ can choose a subsequence of $(u_n)$ 
such that (see the proof of Lemma 2.6 in \cite{Kroe10b} for more details)
\begin{equation}\label{ldecloc-2a}
  S_m=\lim_{n\to\infty} \absn{D u_{n}}\big(\Omega\cap (K_1)_{\frac{1}{2m}}\big)  ~~\text{exists}
\end{equation}
and
\begin{align}\label{ldecloc-2}
  \absB{S_m- \absn{D u_{n}}\big(\Omega\cap (K_1)_{\frac{1}{2m}}\big)}
	\leq \frac{1}{n}
	~~\text{for every $n\geq m$.}
\end{align}
From their definition $S_{m}\geq 0$ is non-increasing in $m$, 
and consequently, $S_\infty:=\lim_{m\to \infty}S_m$ exists in $\RR$. 

Now decompose 
$$
	u_{n}=u_{1,n}+u_{2,n},~\text{where $u_{1,n}:=v_{n}$ 
	and $u_{2,n}:=u_{n}-v_{n}$}. 
$$
Clearly, the first line of (i) is satisfied by construction and  the second line of (i) is a consequence of \eqref{ldecloc-1a}, and $u_{j,n}\to 0$ in $L^1$ as $n\to\infty$ just like $u_{n}$. 
To see that $(Du_{2,n})$ does not charge $K_1$ as claimed in (ii), consider the following: 
\begin{align*}	
	\absn{D u_{2,n}}\big(\Omega\cap (K_1)_{\frac{1}{2m}}\big)
	&\leq 
	\int_{\Omega \cap (K_1)_\frac{1}{2m}\cap	(K_1)_\frac{1}{n}\setminus (K_1)_\frac{1}{2n}} \!\!\!\!\!\!
	\absb{u_{n}\otimes \nabla \varphi_n} dx
+\absn{D u_{n}}\big(\Omega \cap (K_1)_\frac{1}{2m}\setminus (K_1)_\frac{1}{2n}\big)\\
	\\
	&=:I_1(n,m)+I_2(n,m).
\end{align*}
It suffices to show that $\sup_{n\in\NN}I_j(n,m)\to 0$ as $m\to \infty$ for $j=1,2$. Observe that $I_2(n,m)=0$ whenever $n\leq m$, and
for every $m\in \NN$,
\begin{align*}
	\sup_{n>m}I_2(n,m)
	&=\sup_{n>m}\big[\absn{D u_{n}}\big(\Omega\cap (K_1)_{\frac{1}{2m}}\big)-\absn{D u_{n}}\big(\Omega\cap (K_1)_{\frac{1}{2n}}\big)\big]\\
	 &\leq \sup_{n>m} \big(S_m-S_{n}+\frac{2}{n}\big)
	\leq S_m-S_\infty+\frac{2}{m}
\end{align*}
by \eqref{ldecloc-2} and the monotonicity of $S_n$.
Hence, $\sup_{n\in\NN}I_2(n,m)\to 0$ as $m\to \infty$. 
In addition, $I_1(n,m)=0$ whenever $m\geq n$; therefore, 
$\sup_{n\in\NN} I_1(n,m)\leq \int_\Omega\absb{u_m\otimes \nabla \varphi_m} dx \leq \frac{1}{m}$ by \eqref{ldecloc-1a}.
\end{proof}
The component sequences in Lemma~\ref{lem:decloc} have almost pairwise disjoint support, and interactions of the derivatives on any pieces where multiple components overlap are negligible in the limit, essentially due to (ii). This causes local integral functionals to behave asymptotically additive along the decomposition of Lemma~\ref{lem:decloc} as made precise in the proposition below. Its proof relies on the simple fact that if we have a sequence $(\mu_n)_{n\in\NN}\subset\cM(\Omega)$ such that $\int_{\cA_n}\,d\mu_n\to 0$ as $n\to\infty$ for every sequence $(\cA_n)_n$ of Borel subsets of $\Omega$, then $|\mu_n|\to 0$, i.e., $\mu_n$ converges to zero in total variation.
\begin{prop}\label{prop:fdecloc}
Suppose that \eqref{f0}, \eqref{f1} and \eqref{f2} hold. Then for every $v\in \mathrm{BV}(\Omega;\RR^M)$ every decomposition
$u_{n}=u_{1,n}+\ldots+u_{J,n}$ with the properties listed in
Lemma~\ref{lem:decloc}, we have that
\[
	\bald
	\mbf{f}(u_{n}+v)-\mbf{f}(v)-\sum_{j=1}^{J} 
	\big[\mbf{f}\big(u_{j,n}+v\big)-\mbf{f}(v)\big]
	\underset{n\to\infty}{\To}0&\\
	\text{in total variation of measures}&,
	\eald
\]
where for each $u\in \mathrm{BV}(\Omega;\RR^M)$, $\mbf{f}(u)$ is the real-valued measure given by
$$
	d\mbf{f}(u)(x):=df(x,Du)=f(x,\nabla u(x))dx+f^\infty\Big(x,\frac{dDu}{d\abs{Du}}(x)\Big)d\abs{D^su}(x).
$$
In particular, 
\be\label{pfdecloc-0}
	F(u_{n}+v)-F(v)-\sum_{j=1}^{J} 
	\big[F\big(u_{j,n}+v\big)-F(v)\big]
	\underset{n\to\infty}{\To}0.
\ee
\end{prop}
\begin{proof}
As before, it suffices to consider the case $J=2$ as the general case follows inductively.
For any sequence $(\mathcal{A}_n)$ of Borel subsets of $\Omega$, we have to show that
$$
	\int_{\mathcal{A}_n} \Big(df(x,Du_{n}+Dv)-df(x,Dv)-\sum_{j=1}^{2} 
	\big[df\big(x,Du_{j,n}+Dv\big)-df(x,Dv)\big]\Big)
	\underset{n\to\infty}{\To}0.
$$
We decompose 
$$
	\mathcal{A}_n=(\mathcal{A}_n\cap K_1) ~\cup~ E_n ~\cup~ (\mathcal{A}_n\cap K_2\setminus (K_1\cup E_n))
$$ 
with
\[
\bald
	E_n:=\mathcal{A}_n
	\cap \bigcap_{j=1}^2 \Big\{\frac{d\absn{D u_{j,n}}}{d(\absn{D u_{1,n}}+\absn{D u_{2,n}}+\absn{Dv})}\neq 0\Big\}
	\subset \mathcal{A}_n \cap (K_1)_\frac{1}{n}\setminus (K_1)_\frac{1}{2n}. 
\eald
\]
Observe that
$$
\begin{aligned}
	&\int_{\mathcal{A}_n} \Big(df(x,Du_{n}+Dv)-df(x,Dv)
	-\sum_{j=1}^2 \left[df\big(x,Du_{j,n}+Dv\big)-df(x,Dv)\right]\Big)&\\
	&\qquad=\int_{E_n} \Big(df(x,Du_{n}+Dv)-df(x,Dv)
	-\sum_{j=1}^2 \left[df\big(x,Du_{j,n}+Dv\big)-df(x,Dv)\right]\Big),
\end{aligned}
$$
as the terms under the integral cancel out outside $E_n$. Moreover, 
$$
	\abs{Du_{2,n}}(E_n)\underset{n\to\infty}{\To}0,\quad \cL^N(E_n)\underset{n\to\infty}{\To}0 \quad\text{and}\quad
	\abs{Dv}(E_n)\underset{n\to\infty}{\To}0,
$$
the first since $E_n\subset (K_1)_\frac{1}{n}\setminus K_1$ and $Du_{2,n}$ does not charge $K_1$, and the latter two by the fact that $\bigcap_n E_n = \emptyset$. 
Due to \eqref{f1}, this implies that
$$
	\int_{E_n} df\big(x,Du_{2,n}+Dv\big)\underset{n\to\infty}{\To}0~~\text{and}~~\int_{E_n} df(x,Dv)\underset{n\to\infty}{\To}0.
$$
It remains to show that
\begin{align}\label{pfdecloc-10}
	\int_{E_n} df(x,Du_{n}+Dv)-\int_{E_n}df\big(x,Du_{1,n}+Dv\big)
	\underset{n\to\infty}{\To}0;
\end{align}

with $G:\cM(\Omega;\RR^{M\times N})\to\RR$ defined by
\be\label{defG}
	G(\lambda):=\int_\Omega df\Big(x,\lambda\Big),
\ee
our assertion \eqref{pfdecloc-10} can be written as follows:
$$
	G(\chi_{E_n}[D(u_{n}+v)])-G(\chi_{E_n}D(u_{1,n}+v))\underset{n\to\infty}{\To}0.
$$
Since $(Du_{2,n})$ does not charge $K_1$, 
\[
	\chi_{E_n}D(u_{n}+v)-\chi_{E_n}D(u_{1,n}+v)=\chi_{E_n} Du_{2,n} \to 0
\]
in total variation of measures. This concludes the proof, because $G$ is uniformly continuous on bounded subsets of $\cM(\Omega;\RR^{M\times N})$, due to Proposition~\ref{prop:measfunc-ucont} below.
\end{proof}
Regarded superficially, the following lemma is somewhat reminiscent of Reshetnyak's continuity theorem (Theorem 2.39 in \cite{AmFuPa00B}, for instance), but it uses norm topology instead of strict topology. It shows \emph{uniform} continuity of integral functionals on bounded sets of measures, for which there is no equivalent in terms of strict convergence. 
\begin{lem}\label{lem:measfunc-ucont}
Let $g:\overline\Omega\times \RR^{k}\to \RR$ be continuous on $\overline\Omega\times S^{k-1}$ and positively $1$-homogeneous in its second variable. Then the functional $G:\cM(\Omega;\RR^{k})\to\RR$,
$$
	G(\mu):=\int_\Omega g\Big(x,\frac{d\mu}{d\abs{\mu}}\Big)\,d\abs{\mu}
$$
is uniformly continuous on bounded subsets of $\cM(\Omega;\RR^{M\times N})$ with respect to the convergence of measures in total variation: For every pair of sequences $(\mu_n),(\lambda_n)\subset \cM(\Omega;\RR^{M\times N})$ such that
$\abs{\mu_n}(\Omega)$ and $\abs{\lambda_n}(\Omega)$ are bounded,
$$
	\abs{\mu_n-\lambda_n}\to 0~~\text{implies that}~~G(\mu_n)-G(\lambda_n)\to 0.
$$
\end{lem}
\begin{proof}
Below, we abbreviate $\sigma_n:=\absn{\mu_n}+\absn{\lambda_n}$.
By the positive $1$-homogeneity of $g$ and the fact that both $\absn{\mu_n}$ and $\absn{\lambda_n}$ are absolutely continuous with respect to $\sigma_n$, we get that
$$
	G(\mu_n)-G(\lambda_n)=\int_\Omega g\Big(x,\frac{d\mu_n}{d\sigma_n}\Big)-g\Big(x,\frac{d\lambda_n}{d\sigma_n}\Big)\,d\sigma_n.
$$
Let $\eps>0$. If $\abs{\mu_n-\lambda_n}\to 0$, or, equivalently,
$$
	\int_\Omega \absB{\frac{d\mu_n}{d\sigma_n}-\frac{d\lambda_n}{d\sigma_n}}\,d\sigma_n\underset{n\to\infty}{\To} 0,
$$
then for each $\delta>0$, the set 
$$
	E_{n,\delta}:=\mysetr{x\in\Omega}{\absB{\frac{d\mu_n}{d\sigma_n}(x)-\frac{d\lambda_n}{d\sigma_n}(x)}\geq \delta}
$$
satisfies 
\be\label{lmuc-1}
	\sigma_n(E_{n,\delta})\to 0~~\text{as $n\to \infty$}.
\ee
Using that $g$ is uniformly continuous on $\overline{\Omega}\times S^{k-1}$, and thus also uniformly continuous
$\overline{\Omega}\times \overline{B}_1(0)$ by $1$-homogeneity in the second variable, we can choose a suitable $\delta=\delta(\eps)>0$ such that
\be\label{lmuc-2}
	\absB{g\Big(x,\frac{d\mu_n}{d\sigma_n}(x)\Big)-g\Big(x,\frac{d\lambda_n}{d\sigma_n}(x)\Big)}\leq \eps
	\quad\text{for every $x\in \Omega\setminus E_{n,\delta}$}.
\ee
By definition of $\sigma_n$, $\absB{\frac{d\mu_n}{d\sigma_n}(x)}\leq 1$ and $\absB{\frac{d\lambda_n}{d\sigma_n}(x)}\leq 1$ 
for $\sigma_n$-a.e.~$x$. Combining \eqref{lmuc-1} and \eqref{lmuc-2}, we infer that
$$
\bald
	\abs{G(\mu_n)-G(\lambda_n)}
	&\leq\int_\Omega \absB{g\Big(x,\frac{d\mu_n}{d\sigma_n}\Big)-g\Big(x,\frac{d\lambda_n}{d\sigma_n}\Big)}\,d\sigma_n\\
	&\leq \eps \sup_n \sigma_n(\Omega)+2 \norm{g}_{L^\infty(\overline{\Omega}\times \overline{B}_1(0))}\sigma_n(E_{n,\delta})\\
	&\underset{n\to\infty}{\To} \eps \sup_n \sigma_n(\Omega).
\eald
$$
Since $\eps$ was arbitrary and $\big(\sigma_n(\Omega)\big)$ are uniformly bounded, this concludes the proof.
\end{proof}
Lemma~\ref{lem:measfunc-ucont} also holds for more general integrands:
\begin{prop}\label{prop:measfunc-ucont}
Assume that $f$ satisfies \eqref{f0}--\eqref{f2}, and let $G:\cM(\Omega;\RR^{M\times N})\to\RR$ be defined by
$$
	G(\mu):=\int_\Omega df(x, \mu) = \int_\Omega f\Big(x,\frac{d\mu}{dx}\Big)dx+f^\infty\Big(x,\frac{d\mu}{d\abs{\mu}}(x)\Big)d\abs{\mu^s}(x).
$$
Then for every pair of sequences $(\mu_n),(\lambda_n)\subset \cM(\Omega;\RR^{M\times N})$ such that
$\abs{\mu_n}(\Omega)$ and $\abs{\lambda_n}(\Omega)$ are bounded,
$$
	\abs{\mu_n-\lambda_n}\to 0~~\text{implies that}~~G(\mu_n)-G(\lambda_n)\to 0.
$$
\end{prop}
\begin{proof}
Let us rewrite $G(\mu)$ as
$$
	G(\mu)=\int_\Omega (f-f^\infty)\Big(x,\frac{d\mu}{dx}\Big)dx+f^\infty\Big(x,\frac{d\mu}{d\abs{\mu}}(x)\Big)d\abs{\mu}(x).
$$
Since the second term already satisfies the claim by Lemma~\ref{lem:measfunc-ucont}, it suffices to show that the assertion also holds for
$$
	H(\mu):=\int_\Omega (f-f^\infty)\Big(x,\frac{d\mu}{dx}\Big)dx.
$$
To see this, fix some $\eps>0$. By Proposition~\ref{prop:finfty}, we can choose a ball $B_R(0) \in \RR^{M\times N}$ with a suitable radius $R=R(\eps)$ such that
\[
	\abs{f(x,\xi)-f^\infty(x,\xi)}\leq \eps(\abs{\xi}+1)~~\text{for all $(x,\xi)\in \overline{\Omega}\times \RR^{M\times N} \setminus B_R(0)$}.
\]
Further, we find a cut-off function $\varphi_\eps \in C_c^\infty(\RR^{M\times N};[0,1])$ with $\varphi_\eps(\xi)=0$ on $B_R(0)$
and set $h_\eps(x,\xi):=[f(x,\xi)-f^\infty(x,\xi)](1-\varphi_\eps(\xi))$. By writing that $[f(x,\xi)-f^\infty(x,\xi)] = h_\eps(x,\xi) + [f(x,\xi)-f^\infty(x,\xi)]\varphi_\eps(\xi))$, we get that
\[
	\abs{H(\mu_n)-H(\lambda_n)}
	\leq 
	\int_\Omega \absB{h_\eps\Big(x,\frac{d\mu_n}{dx}\Big)-h_\eps\Big(x,\frac{d\lambda_n}{dx}\Big)}dx
	+\eps \big(\sup_n\absn{\mu_n}(\Omega)+\sup_n\absn{\lambda_n}(\Omega)+2\big). 
\]
Since $h_\eps$ is supported on a compact subset of $\overline{\Omega}\times \RR^{M\times N}$ and thus uniformly continuous,
and $\norm{\frac{d\mu_n}{dx}-\frac{d\lambda_n}{dx}}_{L^1(\Omega)}\leq \abs{\mu_n-\lambda_n}(\Omega)\to 0$,
dominated convergence yields that 
\[
	\limsup_{n\to \infty} \abs{H(\mu_n)-H(\lambda_n)}
	\leq \eps C,
\]
for arbitrary $\eps>0$.
\end{proof}

\section{A characterization of weak lower semicontinuity}\label{sec:necsuff}
%
Within this section, we prove Theorem~\ref{thm:BVwlsc}. Our starting point is the following result of \cite{KriRi10b} 
on weak$^*$ lower semicontinuity along sequences with fixed boundary values:
\begin{theorem}[adapted from {\cite[Theorem 2]{KriRi10b}}]
\label{WLSC-Dir}
Suppose that $\Omega$ is a bounded Lipschitz domain and that \eqref{f0}, 
\eqref{f1}, \eqref{f2} hold true.
Suppose further that $\{u_k\}_{k\in \mathbb{N}}$ is a bounded sequence 
in $\mathrm{BV}(\Omega; \mathbb{R}^M)$ such that
$u_k \wstar u$ in $\mathrm{BV}(\Omega; \mathbb{R}^M)$ and $u_k=u$ on 
$\partial\Omega$ in the sense of trace.
If $f(x_0, \cdot)$ is quasiconvex for a.e.~$x_0 \in \overline{\Omega}$, then
$$
F(u) \leq \liminf_{k\to \infty} F(u_k),
$$
for the functional $F$ introduced in \eqref{defF}.
\end{theorem}
\begin{rem}
Our functional $F$ does not include the boundary jump term that appears 
in the functional $\mathcal{F}$ in \cite[Theorem 2]{KriRi10b}. However, 
since we assumed that $u_k=u$ on $\partial\Omega$, these terms for $u$ and $u_k$ cancel:
$$
     f^\infty\Big(x,\frac{u_k}{\abs{u_k}}\otimes 
\nu_\Omega\Big)\abs{u_k}=f^\infty\Big(x,\frac{u}{\abs{u}}\otimes 
\nu_\Omega\Big)\abs{u}\quad\text{on $\cH^{N-1}$-a.e.~$x\in\partial\Omega$},
$$
where $\nu_\Omega=\nu_\Omega(x):=-\nu_x$ denotes the inner normal to 
$\partial\Omega$ at $x\in \partial\Omega$.
\end{rem}
At this stage, it is not perfectly clear if Theorem~\ref{WLSC-Dir} stays 
true if $\partial\Omega$ is not Lipschitz. However,
in our proof of Theorem~\ref{thm:BVwlsc}, it suffices to have a lower 
semicontinuity
result along sequences $(u_k)$ whose derivative does not charge 
$\partial\Omega$ in the sense of Definition~\ref{def:charge}, besides 
having the same trace as the limit $u$. For such sequences, we can avoid 
assuming any kind of regularity of $\partial\Omega$:
\begin{cor}
\label{cor:WLSC-Dir}
If $u_k=u$ in a neighborhood of $\partial\Omega$ (possibly depending on $k$) and $(Du_k)_{k\in\NN}$ does not charge $\partial \Omega$, then Theorem~\ref{WLSC-Dir} 
holds even if $\Omega$ is an arbitrary bounded domain with possibly 
irregular boundary.
\end{cor}
\begin{proof}
Let $\eps>0$. We will extend a suitable modification of $F$, denoted $F_\eps$, to 
a larger domain $\Omega'\supset \Omega$ with Lipschitz boundary, and 
then apply Theorem~\ref{WLSC-Dir} to $F_\eps$ on $\Omega'$.

Choose a cut-off function $\varphi_\eps\in C_c^\infty(\RR^N;[0,1])$ such 
that $\varphi_\eps(x)=0$ for every 
$x\in\Omega$ with $\dist{x}{\RR^N\setminus \Omega}\leq \eps$ for every and $\varphi_\eps(x)=1$ for every $x \in \Omega$ that satisfies
$\dist{x}{\RR^N\setminus \Omega}\geq 2\eps$ and $\varphi_\eps(x)=0$
For $v\in \mathrm{BV}(\Omega';\RR^M)$, we define
\[
     df_\eps(x,Dv):=\varphi_\eps(x) df(x,Dv).
\]
Notice that $df_\eps$ vanishes in a whole neighborhood of 
$\partial\Omega$. We therefore can modify its second argument freely in 
this region; in particular, we can extend $u$ to a function $u_\eps$ by setting 
\[
     u_\eps:=\chi_{\Omega_\eps}u \in 
\mathrm{BV}(\Omega';\RR^M),\quad\text{whence}\quad 
df_\eps(\cdot,Du_\eps)=\chi_\Omega df_\eps(\cdot,Du)~~\text{on $\Omega'$,}
\]
where $\Omega_\eps\subset 
\Omega$ is chosen in such a way that it has a smooth boundary (say, of class $C^1$) and $\varphi_\eps=0$ on 
$\Omega\setminus\bar{\Omega}_\eps$.

In addition, for each $k$, the function 
\[
	v_k:=u_k-u
\]
vanishes near $\partial\Omega$ and thus can be extended by zero to a
function in $\mathrm{BV}(\Omega';\RR^M)$. Theorem~\ref{WLSC-Dir}, applied on the 
Lipschitz domain $\Omega'$, now gives that
\begin{equation}\label{corwlsci-0}
     \liminf_{k\to\infty} \int_{\Omega'} df_\eps(x,Dv_k+Du_\eps)\geq 
\int_{\Omega'} df_\eps(x,Du_\eps).
\end{equation}
Moreover,
\begin{equation}\label{corwlsci-1}
     \int_{\Omega'} df_\eps(x,Du_\eps)=\int_{\Omega} 
df_\eps(x,Du)\underset{\eps\to 0^+}{\To} \int_{\Omega} df(x,Du)
\end{equation}
by dominated convergence, and since
$(Dv_k+Du)$ does not charge $\partial\Omega$, in view of \eqref{f1}, we 
also have that
\[
\begin{aligned}
     &\sup_k \abs{\int_{\Omega} df_\eps(x,Dv_k+Du)-\int_{\Omega} 
df(x,Dv_k+Du)} \\
     &\qquad\leq C \sup_k\abs{Dv_k+Du}(\Omega\cap 
\{\varphi_\eps<1\})+C\cL^N(\Omega\cap \{\varphi_\eps<1\})
     \underset{\eps\to 0^+}{\To}0,
\end{aligned}
\]
whence
\begin{equation}\label{corwlsci-2}
     \int_{\Omega'} df_\eps(x,Dv_k+Du_\eps)=\int_{\Omega} 
df_\eps(x,Dv_k+Du)\underset{\eps\to 0^+}{\To} \int_{\Omega} 
df(x,Dv_k+Du)\quad\text{uniformly in $k$}.
\end{equation}
Combined, \eqref{corwlsci-0}--\eqref{corwlsci-2} yield that $F(u) \leq 
\liminf_{k\to \infty} F(v_k+u)=\liminf_{k\to \infty} F(u_k)$, as asserted.
\end{proof}

%
%

Next, we study lower semicontinuity along pure concentrations at the boundary. For such sequences, 
it is always possible to add or remove a non-zero weak$^*$ limit:
\begin{prop}\label{prop:asympaddconc}
Assume that \eqref{f0}--\eqref{f2} hold, and let $(u_n)_{n \in \mathbb{N}}\subset \mathrm{BV}(\Omega;\RR^M)$ be a bounded sequence that is purely concentrating on the boundary, i.e., 
$S_n:=\{u_n\neq 0\}\cup \supp \abs{D u_n} \subset (\partial \Omega)_{r_n}$ with a decreasing sequence $r_n\searrow 0$.
Then for every $u\in \mathrm{BV}(\Omega;\RR^M)$,
\[
	F(u+u_n)-F(u)-F(u_n)+F(0)\underset {n\to\infty}{\To} 0.
\]
Here, $(\partial \Omega)_{r_n}$ denotes the open $r_n$-neighborhood of $\partial \Omega$ as before.
\end{prop}
\begin{proof}
We prove a stronger result, namely that
\[
	f(\cdot,Du+Du_n)-f(\cdot,Du)-f(\cdot,Du_n)+f(\cdot,0)\underset {n\to\infty}{\To} 0
\]
as measures. The four terms on the left hand side cancel outside the set $S_n$ where $Du_n=0$. Therefore, it suffices to show that
\[
	f(\cdot,Du_n+\chi_{S_n}Du)-f(\cdot,Du_n)\to 0\quad\text{and}\quad
	f(\cdot,\chi_{S_n}Du)-f(\cdot,0)\to 0.
\]
This is a consequence of Proposition~\ref{prop:measfunc-ucont}, because $\abs{D u}(S_n)\leq \abs{D u}(\Omega\cap (\partial \Omega)_{r_n}) \to \abs{D u}(\emptyset)= 0$ by monotone convergence.
\end{proof}
As already mentioned, a sufficient condition for lower semicontinuity along pure concentrations at the boundary is boundedness of the functional (not necessarily the integrand) from below:
\begin{prop}\label{prop:lscbndconc}
Assume that \eqref{f0}--\eqref{f2} hold, let $\cU\subset \mathrm{BV}(\Omega;\RR^N)$ be an additively closed set,
and suppose that $F:\cU\to \RR$ is bounded from below. Then 
$F$ is lower semicontinuous along all sequences $(u_n)_{n \in \mathbb{N}}\subset \cU$ that are bounded in $\mathrm{BV}(\Omega;\RR^N)$
and satisfy 
$S_n\subset (\partial \Omega)_{r_n}$ with a decreasing sequence $r_n\searrow 0$, where
$S_n:=\{u_n\neq 0\}\cup \supp \abs{D u_n}$.  
\end{prop}
\begin{proof}
Similarly to the corresponding result in $W^{1,p}$ (cf. Proposition 3.3 in~\cite{Kroe10b}), the proof relies on finding an almost-minimizer in $\cU$.
Indeed, since $F$ is bounded from below, we may choose for every $\eps>0$ a $u^*\in \cU$ such that 
$$
	F(u^*)-\eps \leq I:=\inf\mysetb{F(u)}{u\in \cU}.
$$
By Proposition~\ref{prop:asympaddconc},
\[
	F(u^*+u_n)-F(u^*)-F(u_n)+F(0) = \big(F(u^*+u_n)-F(u_n)\big) + \big(F(0)-F(u^*)\big)\underset{n\to\infty}{\To} 0.
\]
By definition of $u^*$, we conclude that
$$	
	\liminf F(u_n)-F(0)=\liminf F(u^*+u_n)-F(u^*) \geq I-(I+\eps)=-\eps
$$
for every $\eps>0$.
\end{proof}
\begin{rem}
If $W^{1,1}$ is dense in 
$\cU$ 
with respect to area-strict convergence in $\mathrm{BV}$ ($\langle \cdot \rangle$-strict convergence in the notation of \cite{KriRi10a}), 
in particular if $\cU=\mathrm{BV}$, 
$u^*$ can be chosen in $W^{1,1}(\Omega;\RR^n)$ because $F$ is area-strictly-continuous \cite{RiSha13a} (see also \cite[Theorem 5]{KriRi10a} for a special case). In this case, $\chi_{S_n} \nabla u^*\to 0$ in $L^1$, and the proof of Proposition~\ref{prop:lscbndconc} still works even if we only assume that $\cL^N(S_n)\to 0$ (i.e. $(u_n)$ is purely concentrating, but not necessarily at the boundary).
\end{rem}

\begin{proof}[Proof of Theorem~\ref{thm:BVwlsc}]
 The proof is divided into two steps.

\noindent
{\it Step~1.} (Necessity): The necessity of condition (i) can be shown by taking e.g. non-concentrating sequences in $W^{1,p}(\Omega; \mathbb{R}^n)$ that have zero boundary conditions.  We show necessity of (ii). Assume, without loss of generality, that $f(x,0)=0$ for all $x\in\overline{\Omega}$ and that there is $x_0\in\partial\Omega$ such that $f(x_0,\cdot)$ is not of quasi-sublinear growth from below.  In view of Proposition~\ref{prop:qslb}, this means that 
$f^\infty(x_0,\cdot)$ is not of quasi-sublinear growth from below, which means that \eqref{qslb-unfrozeninfty} cannot be satisfied.
Consequently, we have some $\varepsilon>0$ such that for every $n\in\NN$ there is $v_n\in W^{1,1}(B_{1/n}(x_0)\cap\Omega;\RR^M)$ with $v_n=0$ near $\partial B_{1/n}(x_0)$ and 

\begin{align}\int_{\Omega\cap B_{1/n}(x_0)} f^\infty(x,\nabla v_n(x))\,dx
	<	-\eps \int_{\Omega\cap B_{1/n}(x_0)} \abs{\nabla v_n(x)}\,dx \ . 
\end{align}
In particular, $v_n\neq 0$, and we get for $u_n:=v_n/\| v_n\|_{W^{1,1}(B_{1/n}(x_0)\cap\Omega;\RR^{M})}$, extended by zero to the the rest of $\Omega$, that 
	$\| u_n\|_{W^{1,1}(\Omega;\RR^{M})} =1$ and that for all $n$
	\begin{align}\int_{\Omega\cap B_{1/n}(x_0)} f^\infty(x,\nabla u_n(x))\,dx=\int_{\Omega} f^\infty(x,\nabla u_n(x))\,dx<
	-\eps\ . \end{align}
	
	Since the support of $u_n$ shrinks to the point $x_0$, $u_n\stackrel{*}{\rightharpoonup} 0$ in $\mathrm{BV}(\Omega;\RR^M)$, while on the other hand, 
	$$\liminf_{n\to\infty}\int_{\Omega} f^\infty(x_0,\nabla u_n(x))\,dx \leq -\varepsilon< 0\ .$$
	
	Let us find $R>0$ so large that the function $\mu$ in Proposition~\ref{prop:finfty} satisfies 
	\[ 
		\mu(|\xi|)\sup_n(1+\|\nabla u_n\|_{L^1(\Omega;\RR^{M\times N})})= 2\mu(|\xi|)<\frac{\varepsilon}{4}
		~~\text{if $|\xi|>R$}. 
	\]	
	Further, put 
	$$S(n,R):= \overline{\{x\in\Omega\cap B_{1/n}(x_0)|\, |\nabla u_n(x)|\le R\}}\ .$$
	Clearly, $\mathcal{L}^N(S(n,R))\to 0$ as $n\to\infty$. Denote $m(R):=\max_{|A|\le R,x\in\overline{\Omega}}|f(x,A)-f^\infty(x,A)|$ and take $n$ so large that 
	$m(R)\mathcal{L}^N(S(n,R))\le \frac{\varepsilon}{4}$.	
	We get
	\begin{align*}
		&  \int_{B_{1/n}(x_0)\cap\Omega} |f(x,\nabla u_n(x))-f^\infty(x,\nabla u_n(x))|\,dx\\
		&\begin{aligned}[t]
		=~&\int_{S(n,R)}|f(x,\nabla u_n(x))-f^\infty(x,\nabla u_n(x))|\,dx\\
		&+ \int_{B_{1/n}(x_0)\cap\Omega\setminus S(n,R)} |f(x,\nabla u_n(x))-f^\infty(x,\nabla u_n(x))|\, d x\\
		\le ~& m(R)\mathcal{L}^N(S(n,R)) +2\mu(|\xi|)\le \frac{\varepsilon}{2}\ .
		\end{aligned}
	\end{align*}	
	This shows that 
	\[
		\liminf_{n\to\infty}\int_\Omega \big(f(x,\nabla u_n)-f(x,0)\big)\, dx
		=\liminf_{n\to\infty}\int_{B_{1/n}(x_0)\cap\Omega} f(x,\nabla u_n)
		\, dx\leq -\frac{\varepsilon}{2}<0,
	\]
	contradicting weak$^*$ lower semicontinuity of $F$.

\noindent
{\it Step~2.} (Sufficiency):
Let  $(u_n)_{n \in \mathbb{N}}\subset\mathrm{BV}(\Omega;\RR^M)$ be such that such that $u_n \wstar 0$ and let  $u \in \mathrm{BV}(\Omega; \mathbb{R}^M)$ be arbitrary.  We have to prove that 
\begin{equation}
F(u) \leq \liminf_{n \to \infty} F(u+u_n)
\label{prove-aim}
\end{equation}

Also, we may simplify the situation by extracting a subsequence of $(u_n)_{n \in \mathbb{N}}$ (not relabeled) that realizes the liminf in \eqref{prove-aim} so that we may assume that
$$\liminf_{n \to \infty} F(u+u_n)=\lim_{n \to \infty} F(u+u_n).$$ 
Then, to show \eqref{prove-aim}, it suffices to find a (not relabeled) subsequence of $(u_n)_{n \in \mathbb{N}}$ such that 
\begin{equation}
F(u) \le  \lim_{n \to \infty} F(u+u_{n}),
\label{prove-aimReduced}
\end{equation} which we will do in the sequel. 
Without mentioning this explicitly or relabeling, we keep choosing suitable subsequences below whenever necessary (however notice that we will do this finitely many times).

\begin{figure}[htbp] 
 \begin{minipage}{0.5\linewidth} 
  \centering 
  \includegraphics[width=0.9\linewidth]{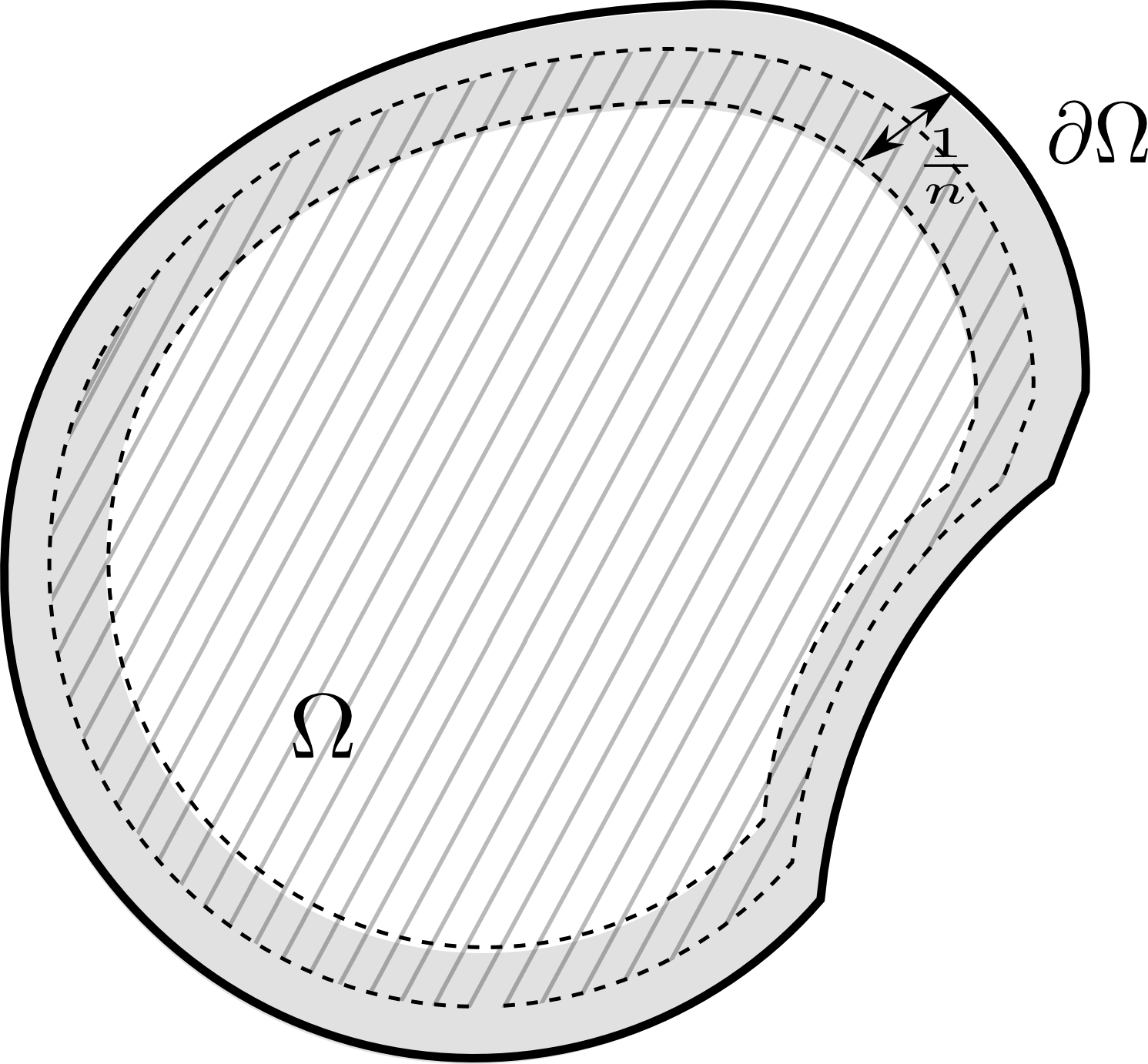} 
  \caption{An illustration of the support of the sequences obtained in (5.8). The support of $(c_n)$ is gray while the support of $(d_n)$ is hatched.} 
  \label{fig:Boundary-Inside} 
 \end{minipage}%
 \hspace*{2ex}
 \begin{minipage}{0.5\linewidth} 
  \centering 
  \includegraphics[width=\linewidth]{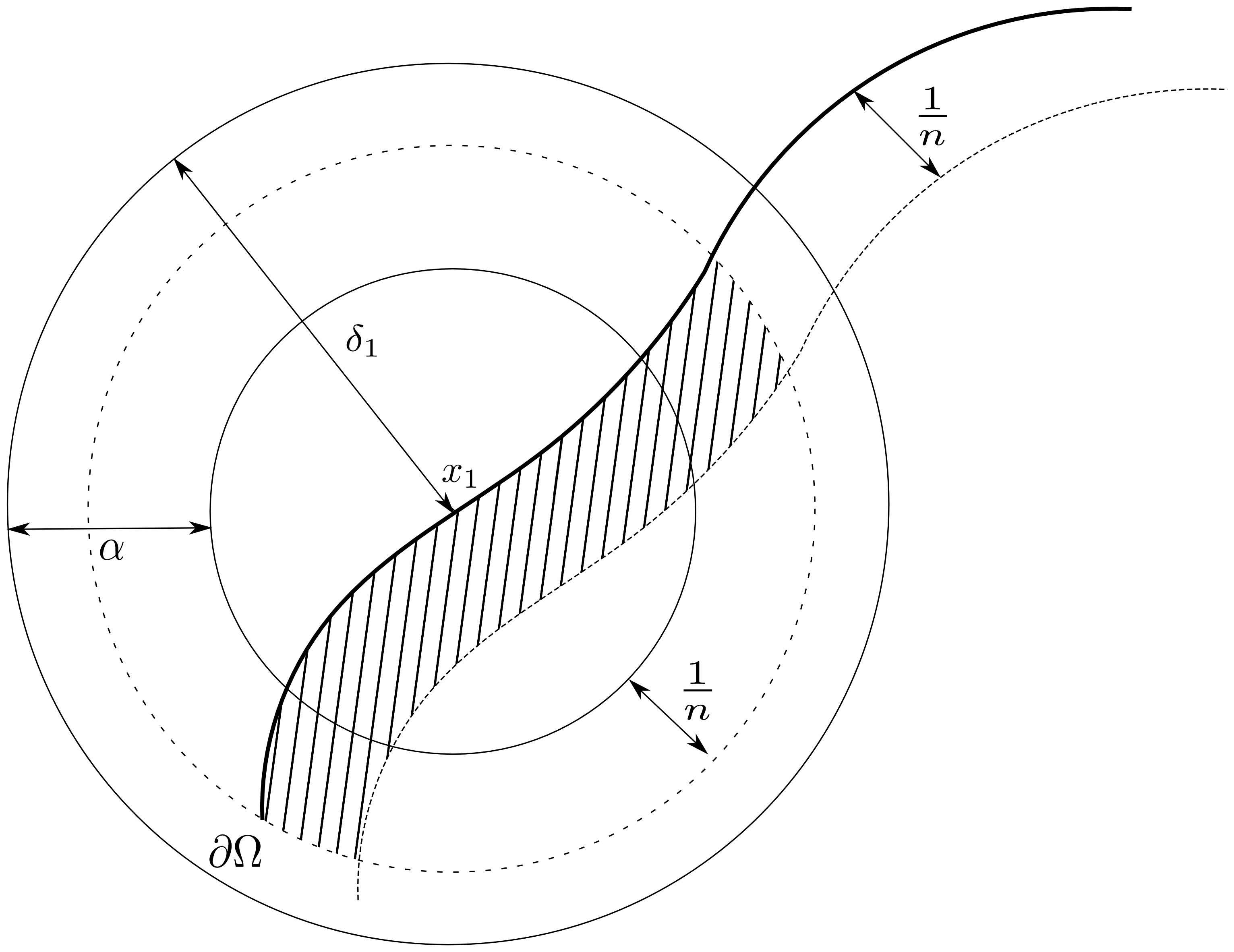} 
  \caption{An illustration of the support of the sequence $(c_{1,n})$ obtained in \eqref{Part-Boundary}. The support of this sequence is hatched.} 
  \label{fig:Boundary_Decom} 
 \end{minipage} 
\end{figure}

To show \eqref{prove-aimReduced}, we first ``separate'' the boundary and interior contributions of $(u_n)_{n \in \mathbb{N}}$. To this end, we use Lemma \ref{lem:decloc} applied to the two compact sets $\partial \Omega$ and $\overline{\Omega}$ an write that 
\begin{equation}
u_n = c_n + d_n
\label{Boundary-Inside_Decom}
\end{equation}
where $(c_n)_{n \in \mathbb{N}}$ is chosen such that $c_n \wstar 0$ in $\mathrm{BV}(\Omega; \mathbb{R}^M)$ and it is supported in $(\partial \Omega)_{\frac{1}{n}}$; i.e. it is a purely concentrating sequence on the boundary. The sequence $(d_n)_{n \in \mathbb{N}}$, on the other hand, is supported in the interior of $\Omega$, i.e. $d_n = 0$ near $\partial \Omega$, does not charge $\partial\Omega$, and also weakly$^*$-converges to $0$. For an illustration of the support of these two sequences we refer the reader to Figure \ref{fig:Boundary-Inside}.

Corollary~\ref{WLSC-Dir} yields that 
\begin{equation}
F(u) \le \lim_{n \to 0} F(u+d_n).
\label{Dn}
\end{equation}

Therefore, let us concentrate on the purely concentrating sequence on the boundary $(c_n)_{n \in \mathbb{N}}$. 
We now use quasi-sublinear growth from below in the form of \eqref{qslb-unfrozenBV}, which is equivalent to  Definition~\ref{def:qslb}
by Remark~\ref{rem:qslbBV} and Proposition~\ref{prop:qslb}.

For fixed $\eps>0$, we cover $\partial \Omega$ by the following collection of balls:
\begin{equation}
\partial \Omega \subset \bigcup_{x \in \partial \Omega} \bigcup_{\delta \leq \tilde{\delta}(x,\eps)} B_\delta (x),
\label{coverOm}
\end{equation}
where $\tilde{\delta}(x,\eps)$ is any such radius for which \eqref{qslb-unfrozenBV} holds; here we recall that if this condition holds with the ball of radius $\tilde{\delta}(x,\eps)$ it also holds for any ball of smaller radius.

Further, since $\partial \Omega$ is a compact we can chose from the cover in \eqref{coverOm} a finite subcover
$$
\partial \Omega \subset \bigcup_{j=1}^J B_{\delta_j} (x_j)
$$
with the radii bounded from below, i.e. $\delta_j \geq \delta_0$ for some $\delta_0 = \delta_0(\eps)$. In fact, since $B_{\delta_j} (x_j)$ are open and the collection is finite, we may still find $\alpha > 0$ so that balls of the radii $\delta_j - \alpha$ still cover $\partial \Omega$; i.e.
$$
\partial \Omega \subset \bigcup_{j=1}^J \overline{B_{\delta_j-\alpha} (x_j)}.
$$

Let us now apply the local decomposition Lemma \ref{lem:decloc} to the sequence $(c_n)_{n \in \mathbb{N}}$ with the compact sets
\begin{align*}
K_1 &= \overline{B_{\delta_1-\alpha} (x_1)} \cap \overline{\Omega}, \\
&\vdots \\
K_J &= \overline{B_{\delta_J-\alpha} (x_J)} \cap \bar{\Omega}, \\
K_{J+1} &= \overline{\Omega \setminus \bigcup_{j=1}^J \overline{B_{\delta_j-\alpha} (x_j)}};
\end{align*}
so we can write 
\begin{equation}
c_n = c_{1,n} + c_{2,n} + \ldots +c_{J+1,n},
\label{Part-Boundary}
\end{equation}
where $c_{j,n}$ are supported in $B_{\delta_j} (x_j)$ for $j=1\ldots J$ is  and $c_{J+1,n}$ is supported in $\Omega$. Notice that we need $n$ large enough compared to $\alpha$ and $\delta_0$ in order to fulfill these requirements; cf. also Figure \ref{fig:Boundary_Decom} for an illustration of the support of $c_{1,n}$. Moreover, $c_{1,n} \ldots c_{J,n}$ retain the property of the original sequence to be concentrating on the boundary while $c_{J+1,n}=0$ for large $n$,  and so
\begin{equation}
F(0) = \lim_{n \to \infty} F(c_{J+1,n}).
\label{Cj+1}
\end{equation}

Further, we define the auxiliary functionals
\begin{align*}
	&G_j(v):=\int_{\Omega\cap B_{\delta_j} (x_j)} 
	df(x,Dv) + \eps \abs{Dv} (\Omega\cap B_{\delta_j} (x_j)),\\
	&v\in \cU_j:=\Big\{v\in \mathrm{BV}(\Omega\cap B_{\delta_j} (x_j);\RR^n) \text{ with } v=0 \text{ near } \partial B_{\delta_j} (x_j)\Big\}.
\end{align*}
Each is bounded from below due to the given quasisublinear growth from below \eqref{qslb-unfrozenBV}. Therefore, they are lower semicontinuous along sequences purely concentrating on the boundary due to Proposition \ref{prop:lscbndconc}; in particular, $G_j$ is lower semicontinuous along $(c_{j,n})$ (note that indeed $(c_{j,n})$ vanishes near $\partial B_{\delta_j} (x_j)$). As a consequence, 
\begin{align*}
	\lim_{n \to \infty} F(c_{j,n})-F(0)&=
	\lim_{n \to \infty} G_j(c_{j,n})-G(0)-\eps \abs{D c_{j,n}} (\Omega\cap B_{\delta_j} (x_j)) \nonumber \\
	&\geq -\eps \lim_{n \to \infty} \abs{D c_{j,n}} (\Omega\cap B_{\delta_j} (x_j)).
\end{align*}
By \eqref{Cj+1} and Proposition \ref{prop:fdecloc} (which applies to $F$ as well as to $u\mapsto \abs{Du}$), the sum over $j$ yields that
\begin{align}
	\lim_{n \to \infty} F(c_n)-F(0)
	&\geq -\eps \lim_{n \to \infty} \abs{D c_{n}} (\Omega),
	\label{Cjn}
\end{align}
Due to Proposition \ref{prop:fdecloc}, Proposition~\ref{prop:asympaddconc}, \eqref{Dn} and \eqref{Cjn}, we get that
\begin{align*}
	&\lim_{n \to \infty} F(u+u_{n})-F(u)
	\begin{aligned}[t]
	&=\lim_{n\to\infty} [F(u+c_n)-F(u)]+[F(u+d_n)-F(u)]\\
	&=\lim_{n\to\infty} [F(c_n)-F(0)]+[F(u+d_n)-F(u)]\\
	&\ge \lim_{n\to\infty} [F(c_n)-F(0)]\\
	&\geq -\eps \limsup_{n\to\infty} \abs{D u_{n}} (\Omega),
	\end{aligned}
\end{align*}
which implies the assertion since $\eps$ was arbitrary.
\end{proof}

\appendix

\section{Proof of Proposition~\ref{prop:qslb}}


(a) We prove the series of equivalences in two steps; first, we show that 
\begin{equation}
\text{$f(x_0,\cdot)$ is qslb at $x_0~\Rightarrow$ \eqref{qslb-finfty}} \Rightarrow \text{$f^\infty(x_0,\cdot)$ is qslb at $x_0$} \Rightarrow \text{$f(x_0,\cdot)$ is qslb at $x_0$}
\label{seriesOfImpl}
\end{equation}
and, in the second step, we proove that $\eqref{qslb-finfty}~\Leftrightarrow~\eqref{qslb-unfrozen}$. 

As for the first implication in \eqref{seriesOfImpl}, we take  $t \geq 0$ and some $v\in W^{1,1}(\Omega\cap B_\delta(x_0);\RR^M)$ with $v=0$ near $\partial B_\delta(x_0)$ so that the quasi-sublinear growth from below of $f(x_0,\cdot)$ implies
$$
\bald
	0 \leq &\frac{1}{t}\left(\int_{B_\delta(x_0)\cap\Omega} f(x_0, t \nabla v(x))+\eps\abs{t\nabla v(x)}\,dx+C\right) \\
	& \qquad \underset{n\to\infty}{\To} \int_{B_\delta(x_0)\cap\Omega} f^\infty(x_0, \nabla v(x))+\eps\abs{\nabla v(x)}\,dx,
\eald
$$
where the limit passage is due to Proposition~\ref{prop:finfty}. The second implication in \eqref{seriesOfImpl} is trivial. The third follows again from Proposition~\ref{prop:finfty}: for some arbitrary $\eps>0$, we fix $h_\eps\geq 0$ (according to this proposition) such that
$$
	\abs{f(x_0,\xi)-f^\infty(x_0,\xi)}\leq \frac{\eps}{2}(1+\abs{\xi})\quad\text{for every $\xi\in \RR^{M\times N}$ with $\abs{\xi}\geq h_\eps$},
$$
Then, we infer that
\[
\bald
	&\int_{B_\delta(x_0)\cap\Omega} f^\infty(x_0, \nabla v(x))\,dx = \int_{B_\delta(x_0)\cap\Omega \cap \{x\in \Omega \,|\, \nabla v(x) > h_\eps\}} \!\!\!\!\!\!\!\!\!\!\!\!\!\!\!\!\!\!\!\!\!\!\!\!\!\!\!\!\!\!\!\!\!\!\!\!f^\infty(x_0, \nabla v(x))\,dx + \int_{B_\delta(x_0)\cap\Omega \cap \{x\in \Omega \,|\, \nabla v(x) \leq h_\eps\}} \!\!\!\!\!\!\!\!\!\!\!\!\!\!\!\!\!\!\!\!\!\!\!\!\!\!\!\!\!\!\!\!\!\!\!\! f^\infty(x_0, \nabla v(x))\,dx \\
	&\leq 
	\int_{B_\delta(x_0)\cap\Omega} \big(f(x_0, \nabla v(x))+\frac{\eps}{2}\abs{\nabla v(x)}+\frac{\eps}{2}\big)\,dx+\abs{B_\delta(x_0)\cap\Omega} \max_{\abs{\xi}\leq h_\eps} f^\infty(x_0, \xi).
\eald
\]
Hence, the integral inequality in Definition~\ref{def:qslb} for $g=f^\infty(x_0, \cdot)$ implies that
\[
\bald
	0&\leq 
	\int_{B_\delta(x_0)\cap\Omega} \big(f(x_0, \nabla v(x))+\eps\abs{\nabla v(x)}\big)\,dx+C_1,
\eald
\]
where $C_1:=C+\abs{B_\delta(x_0)\cap\Omega} (\eps+\max_{\abs{\xi}\leq h_\eps} f(x_0, \xi))$.

As for the second sted (the equivalence $\eqref{qslb-finfty}~\Leftrightarrow~\eqref{qslb-unfrozen}$), we first proceed similarly as in the first step to realize that \eqref{qslb-unfrozen} holds if and only if
\be \label{qslb-unfrozeninfty}
\bald
	&\bald[t]
	&\text{for every $\eps>0$, there exists $\delta>0$ such that}\\
	&\int_{\Omega\cap B_\delta(x_0)} f^\infty(x,\nabla v(x))\,dx
	~\geq~
	-\eps \int_{\Omega\cap B_\delta(x_0)} \abs{\nabla v(x)}\,dx 
	\eald &\\
	&\text{for every $v\in W^{1,1}(\Omega\cap B_\delta(x_0);\RR^M)$ with $v=0$ near $\partial B_\delta(x_0)$}.
\eald
\ee
The only difference between \eqref{qslb-unfrozeninfty} and \eqref{qslb-finfty} is that
the first variable of $f^\infty$ is ``frozen'' to $x_0$ in \eqref{qslb-finfty}. 
Moreover, w.l.o.g., $\delta$ can be chosen arbitrarily small in both conditions. Hence,
it suffices to 
show that $f^\infty(x,\cdot)$ can be replaced by $f^\infty(x_0,\cdot)$ with negligible error for $x$ sufficiently close to $x_0$;
more precisely, we want that
for every $\gamma>0$ (say, $\gamma=\frac{\eps}{2}$), there exists $\delta$ such that
\be \label{p:qslb-finftycont}
		\abs{f^\infty(x,\xi)-f^\infty(x_0,\xi)}\leq \gamma\abs{\xi}\quad\text{for every $(x,\xi)\in \Omega\cap B_\delta(x_0)\times \RR^{M\times N}$.}
\ee
This clearly holds since $f^\infty$ is positively $1$-homogeneous in its second variable and uniformly continuous on 
the compact set $\overline{\Omega}\times S^{MN-1}$.

(b) It is well known(see Remark 2.2 (ii) in \cite{FoMue93a}) that quasiconvexity of $f(x_0, \cdot)$ at zero implies the same for the recession function. 


Moreover, for interior points it is easy to see that if 
$f^\infty(x_0,\cdot)$ is quasiconvex at $0$, it also satisfies \eqref{qslb-finfty} with any $\delta$ such that $B_\delta(x_0)\subset \Omega$ (even for $\eps=0$). By (a), this implies that $f^\infty(x_0,\cdot)$ is qslb at $x_0$. To see the converse, we start from \eqref{qslb-finfty} with $v\in W^{1,1}(B_{1/\delta}(x_0);\RR^M)$ extended by zero to all of $\RR^N$. We then have that
$$
\int_{B_\delta(x_0)} f^\infty(x_0,\nabla v(x))\,dx
	~\geq~
	-\eps \int_{B_\delta(x_0)} \abs{\nabla v(x)}\,dx.
$$
 Take $\eta\in W_0^{1,1}(B_1(0);\RR^M)$ extended by zero to the full space and define $v(x):=\delta^{N-1}\eta(\delta(x+x_0))$. Then $v\in W^{1,1}(B_{1/\delta}(x_0);\RR^M)$  and by the change of variables and by $1$-homogeneity of $f^\infty$
$$
	\int_{ B_{1}(0)} \big(f^\infty(x_0,\nabla \eta(z))+\eps \abs{\nabla \eta(z)}\big) \,dz~\geq~  0 \ ,
$$
which, by setting $\eps \to 0$, yields the quasiconvexity at $0$.

(c) We first show that $f^\infty(x_0,\cdot)$ is of quasi-sublinear growth from below at $x_0$~~$\Leftrightarrow$~~\eqref{qslb-limit1} with $g:=f^\infty(x_0,\cdot)$ (for $x_0\in\partial\Omega$). Due to (a), it suffices to show the equivalence of
\eqref{qslb-finfty} and \eqref{qslb-limit1} with $g:=f^\infty(x_0,\cdot)$. Essentially, this is based on a change of variables argument. First, we blow up $B_\delta$ to a ball of unit size. The blown up (and translated) set $\frac{1}{\delta}[-x_0+B_\delta(x_0)\cap \Omega]$, in a sense, converges to a half-ball as $\delta\to 0$. This made precise by flattening the boundary near $x_0$. The argument follows the one given in \cite{Kroe10b} and is even slightly simpler  since we may exploit $1$-homogeneity of $f^\infty(x_0,\cdot)$. 

Suppose that \eqref{qslb-finfty} holds and fix some $\eps>0$ as well as the associated $\delta=\delta(\eps)>0$. Take some arbitrary $0<r\leq \delta$ and
$v\in W^{1,1}(\Omega\cap B_r(x_0);\RR^M)$ with $v=0$ near $\partial B_r(x_0)$;
a change of variables exploiting the $1$-homogeneity of $f^\infty$
gives that
\be\label{p:qslb-21}
\bald
	\int_{\frac{1}{r}(\Omega-x_0)\cap B_{1}(0)} \big(f^\infty(x_0,\nabla \eta(z))+\eps \abs{\nabla \eta(z)}\big) \,dz
	~\geq~
	0,& \\
	\text{where $\eta(z)=r^{N-1} v\Big(rz+x_0\Big)$.}
\eald
\ee
Moreover, since $\partial\Omega$ is of class $C^1$ near $x_0$, whenever $r$ is small enough, there is a $C^{1}$-diffeomorphism $\Psi_r: \overline{B}_1(0)\to \overline{B}_1(0)$ such that $\Psi(0)=0$,
$\Psi_r(D_{x_0})=\frac{1}{r}(\Omega-x_0)\cap B_1(0)$ and
\be\label{p:qslb-Psiconv}
	\text{$\Psi_r\to \operatorname{id}$ and $\Psi^{-1}_r\to \operatorname{id}$ (the identity) 
	in $C^1(\bar{B}_1(0);\RR^N)$ as $r\to 0^+$.} 
\ee
Changing variables once more to $y=\Psi_r^{-1}(z)$,
we infer that
\be\label{p:qslb-22}
\bald
	\int_{D_{x_0}} \big(f^\infty(x_0,\nabla \varphi(y) (\nabla\Psi_r(y))^{-1})
	+\eps \abs{\nabla \varphi(y) (\nabla\Psi_r(y))^{-1}}\big)
	\abs{\det \nabla\Psi_r(y)}\,dy
	~\geq~
	0,&\\
	\text{where $\varphi=\eta\circ \Psi_r$.}&
\eald
\ee
By \eqref{p:qslb-Psiconv}, using uniform continuity of $f^\infty(x_0,\cdot)$ on the sphere and $1$-homogeneity as in \eqref{p:qslb-finftycont}, 
for each $r$ sufficiently small (independently of $y$ and $\varphi$), we have that
\[
	\abs{f^\infty(x_0,\nabla \varphi(y)(\nabla\Psi_r(y))^{-1})-f^\infty(x_0,\nabla \varphi(y))}\leq \eps \abs{\nabla\varphi(y)},
\]
and analogously
\[
	\absb{\eps\abs{\nabla \varphi(y) (\nabla\Psi_r(y))^{-1}}-\eps\abs{\nabla \varphi(y)}}\leq \eps\abs{\nabla\varphi(y)}.
\]
Plugging this into \eqref{p:qslb-22}, we conclude that
\be\label{p:qslb-23}
\bald
	&\int_{D_{x_0}} \big(f^\infty(x_0,\nabla \varphi(y))
	+3\eps \abs{\nabla \varphi(y)}\big)
	\abs{\det \nabla\Psi_r(y)}\,dy
	~\geq~
	0.
\eald
\ee
Finally, we use that $\det \nabla\Psi_r(y)\to 1$ as $r\to 0$ uniformly in $y$. Due to the linear growth of $f^\infty$, this implies that
for $r$ small enough (independently of $y$ and $\varphi$),
\[
	\abs{(\abs{\det \nabla\Psi_r(y)}-1) f^\infty(x_0,\nabla \varphi(y))} \leq \eps \abs{\nabla \varphi(y)}
\]
and
\[
	\abs{3(\abs{\det \nabla\Psi_r(y)}-1) } \leq 1.
\]
Consequently, 
\be \label{p:qslb-24}
\bald
	&\int_{D_{x_0}} \big(f^\infty(x_0,\nabla \varphi(y))
	+5\eps \abs{\nabla \varphi(y)}\big)
	~\geq~
	0,
\eald
\ee
i.e., the estimate in \eqref{qslb-limit1} with $g:=f^\infty(x_0,\cdot)$ holds
(with $5\eps$ in place of $\eps$, but of course, $\eps>0$ is arbitrary).
Also note that for each 
$\varphi\in W^{1,1}(B_1(0);\RR^M)$ with compact support in $B_1(0)$ (a dense subclass of $W_0^{1,1}(B_1(0);\RR^M)$), the associated function $v$ is given by
$$
	v(x)=r^{1-N}\varphi\Big(\Psi_r^{-1}\Big(\frac{x-x_0}{r}\Big)\Big),
$$
which is admissible in \eqref{qslb-finfty}. Hence, \eqref{qslb-finfty} implies \eqref{qslb-limit1} with $g=f^\infty(x_0,\cdot)$. 

For the converse, first observe that again due to $1$-homogeneity, \eqref{qslb-limit1} with $g=f^\infty(x_0,\cdot)$ has to hold with $C=0$. The rest of the argument essentially amounts to retracing the steps of the calculation above; we omit the details.

\eqref{qslb-limit1} with $g:=f^\infty(x_0,\cdot)$~~$\Leftrightarrow$~~\eqref{qslb-limit2} (for $x_0\in\partial\Omega$): 
We only have to justify that $C=\eps=0$ is admissible in \eqref{qslb-limit1}. As already mentioned above, and similarly in the proof of (a), for each $\eps$, the inequality in \eqref{qslb-limit1} with $g:=f^\infty(x_0,\cdot)$
can only be true for all test functions if it holds with $C=0$, since both $f^\infty(x_0,\cdot)$ and the modulus are positively $1$-homogeneous. Once $C$ is gone, one can pass to the limit as $\eps\to 0$.

\smallskip
\subsection*{Acknowledgments:}  The support by CZ01-DE03/2013-2014/DAAD-56269992 (PPP program) is acknowledged. Moreover,  BB and MK were partly supported by grants GA\v{C}R P201/10/0357 and\\
P107/12/0121, respectively. 

\bibliographystyle{plain}
\bibliography{13_bib}

\begin{thebibliography}{10}

\bibitem{AcFu84a}
E.~Acerbi and N.~Fusco.
\newblock Semicontinuity problems in the calculus of variations.
\newblock {\em Arch.~Ration.~Mech.~Anal.}, 86:125--145, 1984.

\bibitem{AmFuPa00B}
L.~Ambrosio, N.~Fusco, and D.~Pallara.
\newblock {\em Functions of bounded variation and free discontinuity problems}.
\newblock Oxford Mathematical Monographs. Clarendon Press, Oxford, 2000.

\bibitem{BaCheMaSa13a}
M.~{Ba{\'\i}a}, M~{Chermisi}, J.~{Matias}, and P.M. {Santos}.
\newblock {Lower semicontinuity and relaxation of signed functionals with
  linear growth in the context of $\mathcal A$-quasiconvexity.}
\newblock {\em {Calc. Var. Partial Differ. Equ.}}, 47(3-4):465--498, 2013.

\bibitem{BaMa84a}
J.M. Ball and J.E. Marsden.
\newblock Quasiconvexity at the boundary, positivity of the second variation
  and elastic stability.
\newblock {\em Arch. Ration. Mech. Anal.}, 86:251--277, 1984.

\bibitem{BeSc13a}
L.~Beck and T.~Schmidt.
\newblock On the {D}irichlet problem for variational integrals in {$BV$}.
\newblock {\em J. Reine Angew. Math.}, 674:113--194, 2013.

\bibitem{FoMue92a}
I.~Fonseca and S.~M{\"u}ller.
\newblock Quasi-convex integrands and lower semicontinuity in {$L^1$}.
\newblock {\em SIAM J. Math. Anal.}, 23(5):1081--1098, 1992.

\bibitem{FoMue93a}
I.~Fonseca and S.~{M\"uller}.
\newblock {Relaxation of quasiconvex functionals in $BV(\Omega, {\mathbb R}\sp
  N)$ for integrands $f(x, u, \bigtriangledown u)$.}
\newblock {\em Arch. Ration. Mech. Anal.}, 123(1):1--49, 1993.

\bibitem{FoMuePe98a}
I.~Fonseca, S.~M{\"{u}}ller, and P.~Pedregal.
\newblock Analysis of concentration and oscillation effects generated by
  gradients.
\newblock {\em SIAM J.~Math.~Anal.}, 29(3):736--756, 1998.

\bibitem{KaKroeKru12a}
A.~Ka{\l}amajska, S.~Kr{\"o}mer, and M.~{Kru\v{z}\'{\i}k}.
\newblock Sequential weak continuity of null lagrangians at the boundary.
\newblock {\em Calculus of Variations and Partial Differential Equations},
  49:1263--1278, 2014.

\bibitem{KaKru08a}
A.~Ka{\l}amajska and M.~{Kru{\v{z}\'\i}k}.
\newblock Oscillations and concentrations in sequences of gradients.
\newblock {\em ESAIM, Control Optim. Calc. Var.}, 14(1):71--104, 2008.

\bibitem{KriRi10b}
J.~Kristensen and F.~Rindler.
\newblock Characterization of generalized gradient {Y}oung measures generated
  by sequences in {$W^{1,1}$} and {BV}.
\newblock {\em Arch. Ration. Mech. Anal.}, 197(2):539--598, 2010.

\bibitem{KriRi10a}
J.~Kristensen and F.~Rindler.
\newblock Relaxation of signed integral functionals in {BV}.
\newblock {\em Calc. Var. Partial Differ. Equ.}, 37(1-2):29--62, 2010.

\bibitem{Kri94a}
Jan Kristensen.
\newblock Finite functionals and {Young} measures generated by gradients of
  {Sobolev} functions.
\newblock Mat-report 1994-34, Math.~Institute, Technical University of Denmark,
  1994.

\bibitem{KroeKru13a}
S.~Kr{\"o}mer and M.~{Kru\v{z}\'{\i}k}.
\newblock Oscillations and concentrations in sequences of gradients up to the
  boundary.
\newblock {\em Journal of Convex Analysis}, 20(3):723--752, 2013.
\newblock Preprint version: arXiv:1109.3020.

\bibitem{Kroe10b}
Stefan Kr{\"{o}}mer.
\newblock On the role of lower bounds in characterizations of weak lower
  semicontinuity of multiple integrals.
\newblock {\em Adv. Calc. Var.}, 3(4):387--408, 2010.

\bibitem{Kru10a}
Martin {Kru{\v{z}\'\i}k}.
\newblock Quasiconvexity at the boundary and concentration effects generated by
  gradients.
\newblock {\em ESAIM Control Optim. Calc. Var.}, 19:679--700, 2013.

\bibitem{MieSpre98a}
A.~Mielke and P.~Sprenger.
\newblock Quasiconvexity at the boundary and a simple variational formulation
  of {Agmon}'s condition.
\newblock {\em J. Elasticity}, 51(1):23--41, 1998.

\bibitem{Mo52a}
Charles~B. Morrey.
\newblock Quasi-convexity and the lower semicontinuity of multiple integrals.
\newblock {\em Pac. J. Math.}, 2:25--53, 1952.

\bibitem{RiSha13a}
F.~Rindler and G.~Shaw.
\newblock Strictly continuous extensions and convex lower semicontinuity of
  functionals with linear growth.
\newblock Preprint arXiv:1312.4554v2 [math.AP], 2013.

\bibitem{Spre96B}
Pius Sprenger.
\newblock {\em {Quasikonvexit\"at am Rande und Null-Lagrange-Funktionen in der
  nichtkonvexen Variationsrechnung}}.
\newblock PhD thesis, {Universit\"at Hannover}, 1996.

\end{thebibliography}
\end{document}